\documentclass[11pt]{amsart}
\usepackage{amsfonts} 
\usepackage{amssymb}
\usepackage{amsthm}
\usepackage{amsmath} 
\usepackage{mathrsfs}
\usepackage{mathtools}
\usepackage{dsfont}
\usepackage{esint}
\usepackage{ulem}
\usepackage{marginnote}
\usepackage{array}
\usepackage{algorithm}
\usepackage{algpseudocode}
\usepackage{soul}
\usepackage{comment}

\usepackage{booktabs} 

\usepackage{MnSymbol}

\usepackage{longtable}

\usepackage[all]{xy}

\usepackage[english]{babel} 
\usepackage[left=3.5cm,right=3.5cm,top=3.5cm,bottom=3cm,headsep=0.7cm]{geometry}
\usepackage{xargs}
\usepackage[pdftex,dvipsnames]{xcolor}
\usepackage[colorinlistoftodos,prependcaption,textsize=tiny]{todonotes}
\newcommandx{\unsure}[2][1=]{\todo[linecolor=red,backgroundcolor=red!25,bordercolor=red,#1]{#2}}
\newcommandx{\change}[2][1=]{\todo[linecolor=blue,backgroundcolor=blue!25,bordercolor=blue,#1]{#2}}
\newcommandx{\info}[2][1=]{\todo[linecolor=OliveGreen,backgroundcolor=OliveGreen!25,bordercolor=OliveGreen,#1]{#2}}
\newcommandx{\improvement}[2][1=]{\todo[linecolor=Plum,backgroundcolor=Plum!25,bordercolor=Plum,#1]{#2}}
\newcommandx{\thiswillnotshow}[2][1=]{\todo[disable,#1]{#2}}

\usepackage{pgf,tikz}
\usetikzlibrary{positioning, calc, arrows.meta}
\usetikzlibrary{arrows}
\usepackage{subcaption}
\usepackage{graphicx}
\usepackage[parfill]{parskip}

\usepackage{enumerate}
\usepackage{enumitem}

\usepackage{hyperref}
\usepackage[nameinlink]{cleveref}
\usepackage{mathtools}
\hypersetup{
    colorlinks,
    linkcolor={red!80!black},
    citecolor={blue!80!black}
}

\crefname{assumption}{assumption}{assumptions}  
\Crefname{assumption}{Assumption}{Assumptions}  
\crefname{counterexample}{counterexample}{counterexamples}
\Crefname{counterexample}{Counterexample}{Counterexamples}

\theoremstyle{plain}
\begingroup
\theoremstyle{plain}
\newtheorem{theorem}{Theorem}[section]
\newtheorem{corollary}[theorem]{Corollary}
\newtheorem{proposition}[theorem]{Proposition}
\newtheorem{lemma}[theorem]{Lemma}
\theoremstyle{definition}
\newtheorem{definition}[theorem]{Definition}
\newtheorem{assumption}[theorem]{Assumption}

\theoremstyle{remark}
\newtheorem{remark}[theorem]{Remark}

\endgroup

\theoremstyle{definition}
\theoremstyle{remark}

\numberwithin{equation}{section}

\newcommand{\N}{\mathbb{N}}

\mathchardef\emptyset="001F

\newcommand{\eps}{\varepsilon}

\newcommand{\K}{\mathcal{K}}
\newcommand{\X}{\mathcal{X}}
\newcommand{\Y}{\mathcal{Y}}

\newcommand{\C}{\mathbb{C}}

\newcommand{\Tree}{\mathsf{Tree}}

\newcommand{\T}{\mathbb T}

\newcommand{\diam}{\operatorname{diam}}
\newcommand{\supp}{\operatorname{supp}}

\newcommand{\SCIG}{\operatorname{SCI}_{\mathrm G}}

\newcommand{\sigap}{\sigma_{\mathrm{ap}}}
\newcommand{\sigp}{\sigma_{\mathrm p}}

\newcommand{\ev}{\operatorname{ev}}
\newcommand{\concat}{\mathbin{{}^{\smallfrown}}}
\newcommand{\One}{\mathbf 1}
\newcommand{\id}{\mathrm{id}}

\title[\bf Endpoint Koopman Spectral Computation]{\bf Endpoint Koopman Spectral Computation: $L^1$ Residual Bounds, $L^\infty$ Instability, and Point-Spectral SCI Calibration Families}

\begin{document}

\author[C.~Sorg]{Christopher Sorg$^1$}
	\address[C.~Sorg]{
		\textup{Chair for Theoretical computer science, mathematics, and operations research}
		\newline \indent
		\textup{Department of Computer Science} \newline \indent
		\textup{University of the Bundeswehr Munich}
		\newline \indent
		\textup{85577 Neubiberg, Germany}}

\email{{\href{mailto:chr.sorg@unibw.de}{\textcolor{blue}{\texttt{chr.sorg@unibw.de}}}}
}

\footnotetext[1]{Inf1, University of the Bundeswehr Munich, Werner-Heisenberg-Weg 39, 85577 Neubiberg, Germany}

\begin{abstract}
We study endpoint Koopman spectral computation from the viewpoint of the Solvability
Complexity Index (SCI).  Let \((\mathcal X,d)\) be a compact metric space with finite
Borel measure \(\omega\), and let \(\mathcal K_F\) be the Koopman operator associated
with a continuous nonsingular map \(F:\mathcal X\to\mathcal X\).

First, on \(L^1(\mathcal X,\omega)\), we record the endpoint residual upper-bound in the target-split form. The regularized compact fixed-\(\varepsilon\)
target $R_{\mathrm{ap},\varepsilon}(\mathcal K_F)$ is separated from the closed fixed-\(\varepsilon\) target $C_{\mathrm{ap},\varepsilon}(\mathcal K_F)$ and from the exact approximate point spectrum $\sigma_{\mathrm{ap}}(\mathcal K_F).$ This endpoint statement uses the same point-evaluation plus fixed-quadrature information model as the \(1<p<\infty\) residual theory.

Second, we isolate two obstructions at the nonseparable endpoint \(L^\infty\). Fixed quadrature schemes do not discretize the full \(L^\infty\) unit sphere, and even inside measure-preserving Cantor homeomorphisms the map $F\mapsto \sigma_{\mathrm{ap}}(\mathcal K_F:L^\infty\to L^\infty)$ is maximally discontinuous in Hausdorff distance under arbitrarily small uniform perturbations of \(F\). We also show that finite-period Silver-tree block constructions cannot yield analytic hardness for the \(L^\infty\) approximate point spectrum: for a fixed non-torsion \(z_0\in\mathbb T\), the condition $z_0\in\sigma_{\mathrm{ap}}(\mathcal K_{F}:L^\infty\to L^\infty)$ collapses to a Borel unbounded-period condition. In addition, fixed \(L^\infty\) point-eigenvalue membership is Borel in the measure-preserving continuous class, so one fixed eigenvalue cannot encode a non-Borel tree predicate.

Third, we construct Koopman point-spectrum calibration families on the Cantor space. For each \(m\in \mathbb{N}\), we build a family of continuous measure-preserving Cantor homeomorphisms whose labelled exact \(L^\infty\) point-eigenvalue decisions are finite-query equivalent to the canonical alternating Cantor-matrix source problem of raw type-\(G\) height \(m\). Consequently these Koopman decision problems have exact raw type-\(G\) SCI height \(m\), and their tagged disjoint union has raw type-\(G\) SCI $\infty$.
\end{abstract}

\maketitle

\begin{center}\small
\textbf{Keywords:} Solvability Complexity Index; Koopman operators; approximate point spectrum; point spectrum; Cantor dynamics; finite-query reductions; Weihrauch complexity
\end{center}
\tableofcontents

\section{Introduction}\label{sec:introduction}

The solvability complexity index (SCI) measures how many nested limiting processes are
unavoidable when computing a mathematical object from a prescribed information
interface.  It was introduced in spectral computation and has since become a framework
for classifying problems whose numerical solution cannot be captured by a single
finite procedure; see, for example, \cite{hansen2011solvability,ben2015computing,colbrook2022foundations}.  In this
language a computational problem is specified by an input class \(\Omega\), a target
space \(\mathcal M\), a problem map
\[
        \Xi:\Omega\to\mathcal M,
\]
and an evaluation family \(\Lambda\).  A (raw) type-\(G\) tower uses finitely many oracle
values at each deepest level and then takes iterated pointwise limits; its height is the
number of nested limits required. The type-\(A\) and represented-space variants impose
additional effectiveness or regularity restrictions on the deepest-level procedures.

This paper studies endpoint Koopman spectral computation from that point of view.
The operator-theoretic viewpoint goes back to Koopman and Koopman-von Neumann \cite{koopman1931hamiltonian,koopman1932dynamical}; modern data-driven
interest was revived by spectral Koopman methods, DMD/EDMD, and related approaches \cite{mezic2005spectral,rowley2009spectral,schmid2010dynamic,williams2015edmd,
budivsic2012applied,brunton2021modern}. Because Koopman operators are in general
infinite-dimensional and may have continuous spectrum or often have no finite-dimensional
invariant subspaces, verified residual and pseudospectral methods have become particularly important in rigorous data-driven work \cite{colbrook2024rigorous,colbrook2023residual,colbrook2024limits}.

The endpoint spaces \(L^1(\mathcal X,\omega)\) and \(L^\infty(\mathcal X,\omega)\) behave very differently from the reflexive regime \(1<p<\infty\). In the \(1<p<\infty\) theory \cite{sorg2025solvability}, the approximate-point problem is split into the regularized compact target $R_{\mathrm{ap},\varepsilon}(\K_F)$, the closed target $C_{\mathrm{ap},\varepsilon}(\K_F)$, and the exact approximate point spectrum $\sigma_{\mathrm{ap}}(\K_F)$, which will be defined in precise form in \cref{sec:target-split}. The proofs in \cite{sorg2025solvability} use continuous finite-dimensional dictionaries and tagged quadrature residuals. The present work records the corresponding \(L^1\) endpoint abstraction, but the main new phenomena occur at \(L^\infty\).

At \(L^\infty\), the observable space is nonseparable in general, fixed quadrature data do not discretize the full unit sphere, and approximate eigenvectors may live in parts of the space invisible to any fixed finite-dimensional trial mechanism. Thus one has to separate two questions which are often conflated. The first is the residual question: what can be verified from inequalities of the form
\[
        \|(\K_F-zI)g\|\le\varepsilon?
\]
The second is the exact information-theoretic question: what raw SCI height can be realized by Koopman spectral predicates when the input is the underlying map \(F\) and the oracle consists of point evaluations? The exact point spectrum is not a robust numerical target, but it is an interesting calibration object for the raw SCI hierarchy.

We work throughout the main \(L^\infty\) constructions on the Cantor space $\mathcal X=2^{\mathbb N}$ with its standard ultrametric and Bernoulli product measure. This setting is deliberately chosen. It is concrete enough to allow explicit clopen cyclic gadgets and measure-preserving homeomorphisms, but flexible enough to encode Cantor-matrix decision problems by point evaluations. In this way it gives a transparent operator-theoretic realization of the abstract finite-height source problems developed in \cite{sorg2026foundational}.

The contributions of this paper are (i) \textbf{\(L^1\) endpoint residual targets}, that complete the $1<p<\infty$ theory in \cite{sorg2025solvability}; (ii) \textbf{\(L^\infty\) approximate-point no-go diagnostics}, where we prove a block lower-norm formula for \(L^\infty\)-products and use it to analyze finite-period constructions. The resulting finite-period Silver-tree scheme does not realize analytic hardness: for a fixed non-torsion point \(z_0\in\mathbb T\), membership
\[
        z_0\in\sigma_{\mathrm{ap}}(\mathcal K_{F}:L^\infty\to L^\infty)
\]
collapses to the Borel condition that the selected periods are unbounded. We also prove a \(C^0\)-instability result:
\[
        F^{(n)}\to \mathrm{id}_{2^{\mathbb N}}
        \quad\text{uniformly},
        \qquad
        \sigma_{\mathrm{ap}}(\mathcal K_{F^{(n)}})=\mathbb T,
        \qquad
        \sigma_{\mathrm{ap}}(\mathcal K_{\mathrm{id}})=\{1\}.
\]
Thus small uniform perturbations can cause a maximal Hausdorff jump in the \(L^\infty\) approximate point spectrum. (iii) \textbf{No analytic fixed-eigenvalue route for \(L^\infty\) point spectrum:} For measure-preserving maps, fixed \(L^\infty\) point-eigenvalue membership agrees with fixed \(L^2\) point-eigenvalue membership. We prove that for each fixed
\(\lambda\in\mathbb C\) the set
\[
        \{F:F_\#\omega=\omega,\ \lambda\in\sigma_{\mathrm p}(\mathcal K_F:L^\infty\to L^\infty)\}
\]
is Borel in the uniform topology on \(C(\mathcal X,\mathcal X)\). Hence no Borel coding of a non-Borel tree predicate, such as the Silver-tree predicate, can reduce that predicate to one fixed \(L^\infty\) point-eigenvalue question. (iv) \textbf{Point-spectral Koopman calibration families:} The positive result is different. We do not use one fixed eigenvalue. Instead, for each \(m\in \mathbb{N}\), we construct a family of continuous measure-preserving Cantor homeomorphisms and a sequence of labelled roots of unity $(\lambda_j)_{j\ge1}$ such that labelled exact \(L^\infty\) point-eigenvalue membership recovers the bits of a Cantor-matrix input. This yields a finite-query equivalence between the Koopman decision problem and the canonical alternating Cantor-matrix source problem of height \(m\) from \cite{sorg2026foundational}. Consequently the Koopman problem has exact raw type-\(G\) SCI height \(m\). The tagged disjoint union has infinite raw type-\(G\) SCI height.

The message is therefore deliberately two-sided. For numerical Koopman practice, residual and pseudospectral quantities remain the robust targets. Exact point spectrum is brittle and is not presented here as the practical object to compute from data. Its role in this paper is foundational: it provides a concrete Koopman realization of the finite and infinite raw SCI calibration hierarchy developed in \cite{sorg2026foundational}.

\section{Koopman Framework On \texorpdfstring{$L^{1}(\mathcal X,\omega)$}{L1(X,w)} And On \texorpdfstring{$L^{\infty}(\mathcal X,\omega)$}{Loo(X,w)}}
\label{sec:KoopL1Loo}

Throughout this section:
\begin{itemize}
  \item $(\mathcal X,d_{\mathcal X})$ is a \textit{compact} metric space;
  \item $\omega$ is a finite Borel measure on $\mathcal X$ with $\omega(\mathcal X)\neq 0$,
\end{itemize}
since the case $\omega(\mathcal{X}) = 0$ is trivial.

\subsection{Koopman Operator On \texorpdfstring{$L^{1}(\mathcal{X},\omega)$}{L1(X,w)}}
\label{subsec:KoopmanL1}
We write $\|\,\cdot\,\|_{p}:=\|\cdot\|_{L^{p}(\omega)}$ and given a measurable map $F:\mathcal X\to\mathcal X$, define
\[
  \mathcal K_{F} : L^{1}(\mathcal X,\omega)\;\rightarrow\;L^{1}(\mathcal X,\omega),
  \qquad
  (\mathcal K_{F}f)(x) := f\bigl(F(x)\bigr),\;\;f\in L^{1}(\mathcal X,\omega).
\]
We say $F$ is \textit{nonsingular} with respect to $\omega$, if $F_{\#}\omega\ll \omega$ and denote its Radon-Nikodým density by
\[
  \rho_{F} := \frac{dF_{\#}\omega}{d\omega}\in L^{1}(\mathcal{X},\omega).
\]

The boundedness criteria are the exact same as in the $1<p<\infty$ case, since \cite[Proposition 3.1]{sorg2025solvability} holds verbatim for $p=1$. Nevertheless we state it for completeness here again.

\begin{proposition}[Boundedness on $L^{1}$]\label{prop:boundednessL1}
Let $F$ be nonsingular and $\rho_{F}\in L^{\infty}(\omega)$. Then $\mathcal K_{F}:L^{1}(\omega)\to L^{1}(\omega)$ is bounded and
\[
  \|\mathcal K_{F}\|_{L^{1}\!\to L^{1}}\;=\;\|\rho_{F}\|_{L^{\infty}}.
\]
In particular, if $F$ is measure–preserving ($\rho_{F}=1$ a.e.), then $\mathcal K_{F}$ is an isometry on $L^{1}(\omega)$.
\end{proposition}

\begin{proof}
For \(f\in L^1(\omega)\),
\[
        \|\K_Ff\|_{L^1(\omega)} =
        \int_\X |f(F(x))|\,d\omega(x) =
        \int_\X |f(y)|\,d(F_\#\omega)(y) =
        \int_\X |f(y)|\rho_F(y)\,d\omega(y).
\]
Hence
\[
        \|\K_Ff\|_1\le \|\rho_F\|_\infty\|f\|_1.
\]
Thus \(\K_F\) is bounded and
\[
        \|\K_F\|_{1\to1}\le \|\rho_F\|_\infty.
\]

For the reverse inequality, let
\[
        M:=\|\rho_F\|_\infty.
\]
If \(M=0\), there is nothing to prove. For \(\eta>0\), the set
\[
        E_\eta:=\{y\in \X:\rho_F(y)>M-\eta\}
\]
has positive \(\omega\)-measure. Put
\[
        f_\eta:=\frac{1_{E_\eta}}{\omega(E_\eta)}.
\]
Then
\[
        \|f_\eta\|_1=1
\]
and
\[
        \|\K_Ff_\eta\|_1 =
        \frac1{\omega(E_\eta)} \int_{E_\eta}\rho_F\,d\omega \ge M-\eta.
\]
Letting \(\eta\downarrow0\) gives
\[
        \|\K_F\|_{1\to1}\ge M.
\]
Therefore
\[
        \|\K_F\|_{1\to1}=\|\rho_F\|_\infty.
\]
If \(F\) is measure-preserving, then \(\rho_F=1\) a.e., so \(\K_F\) is an isometry.
\end{proof}

Therefore we fix in the $L^1$ case as assumption

\begin{assumption}\label{ass:BoundenessL1}
Suppose the following holds.
\begin{enumerate}[label=\textbf{(LB\arabic*)}]
    \item \label{LB1} $F_{\#}\omega\ll\omega$; \label{ass:B1}
    \item \label{LB2} $\rho_{F}(x) := \dfrac{dF_{\#}\omega}{d\omega}(x) \in L^{\infty}(\mathcal{X},\omega)$.
\end{enumerate}
\end{assumption}

We shall frequently use the following stability estimate.

\begin{corollary}[Composition stability in $L^{1}$]\label{lem:L1-stability}
Under \ref{LB1}–\ref{LB2}, if $f_{n}\to f$ in $L^{1}(\omega)$, then $f_{n}\circ F\to f\circ F$ in $L^{1}(\omega)$ and
\[
  \|f_{n}\circ F-f\circ F\|_{1}
  = \int |f_{n}-f|\,d(F_{\#}\omega)
  \le \|\rho_{F}\|_{\infty}\,\|f_{n}-f\|_{1}.
\]
\end{corollary}

\subsection{Koopman Operator On \texorpdfstring{$L^{\infty}(\mathcal{X},\omega)$}{Loo(X,w)}}
\label{subsec:KoopmanLoo}

We write $\|\,\cdot\,\|_{\infty}:=\|\cdot\|_{L^{\infty}(\omega)}$. Given a measurable map
$F:\mathcal{X}\to\mathcal{X}$ we would like to define, analogously to the $L^{p}$ case,
\[
  \mathcal K_{F} : L^{\infty}(\mathcal X,\omega)\;\rightarrow\;L^{\infty}(\mathcal X,\omega),
  \qquad
  (\mathcal K_{F}f)(x) := f\bigl(F(x)\bigr).
\]
Since $L^{\infty}(\omega)$ consists of $\omega$-a.e.\ equivalence classes, the expression
$f\circ F$ is \textit{well-defined on classes} if and only if changing $f$ on an
$\omega$-null set does not change $f\circ F$ on a set of positive $\omega$-measure.
This is exactly the nonsingularity condition.

\begin{proposition}[Well-definedness and boundedness on $L^{\infty}$]\label{prop:boundednessLoo}
Let $F:\mathcal{X}\to\mathcal{X}$ be measurable. The following are equivalent.
\begin{enumerate}[label=\textbf{(\roman*)}]
  \item The rule $[f] \mapsto [f\circ F]$ defines a well-defined linear operator
  $\mathcal K_F:L^{\infty}(\omega)\to L^{\infty}(\omega)$.
  \item $F$ is nonsingular with respect to $\omega$, i.e.\ $F_{\#}\omega\ll \omega$.
\end{enumerate}
If these equivalent conditions hold, then \(\K_F\) is a contraction, i.e.
\[
\|\K_F f\|_\infty\le \|f\|_\infty.
\]
Since \(\K_F \One = \One\) and \(\omega(\X)>0\), one has \(\|\K_F\|=1\).
\end{proposition}

\begin{proof}
Assume first that \(F_\#\omega\ll\omega\). Let \(f,g\in L^\infty(\omega)\) with \(f=g\) \(\omega\)-a.e. Then
\[
        N:=\{y:f(y)\ne g(y)\}
\]
is an \(\omega\)-null-set. By nonsingularity,
\[
        \omega(F^{-1}(N))=F_\#\omega(N)=0.
\]
Hence
\[
        f\circ F=g\circ F \qquad \omega\text{-a.e.}
\]
Thus the rule \([f]\mapsto[f\circ F]\) is well-defined on equivalence classes.

Moreover, if \(|f|\le M\) \(\omega\)-a.e., then
\[
        |f\circ F|\le M \qquad \omega\text{-a.e.},
\]
because \(F^{-1}(\{|f|>M\})\) is an \(\omega\)-null-set. Therefore
\[
        \|\K_F f\|_\infty\le \|f\|_\infty.
\]
Since \(\K_F \One = \One\) and \(\omega(\X)>0\), we also have
\[
        \|\K_F\|=1.
\]

Conversely, suppose \(F_\#\omega\not\ll\omega\). Then there is a Borel set \(N\subseteq \X\) such that
\[
        \omega(N)=0, \,
        F_\#\omega(N)=\omega(F^{-1}(N))>0.
\]
The function \(\mathbf 1_N\) represents the zero class in \(L^\infty(\omega)\), but
\[
        \mathbf 1_N\circ F= \mathbf 1_{F^{-1}(N)}
\]
is not zero in \(L^\infty(\omega)\). Hence composition does not define a well-defined operator on \(L^\infty\)-classes. This proves the equivalence.
\end{proof}

Therefore we fix in the $L^\infty$ case as assumption

\begin{assumption}\label{ass:BoundenessLoo}
Suppose the following holds.
\begin{enumerate}[label=\textbf{(UB\arabic*)}]
  \item \label{UB1} $F_{\#}\omega\ll\omega$.
\end{enumerate}
\end{assumption}

We shall frequently use the following stability estimate, that follows directly from \cref{lem:L1-stability} and \cref{prop:boundednessLoo}.

\begin{lemma}[Composition stability in $L^{\infty}$]\label{lem:Loo-stability}
Under \ref{UB1}, if $f_{n}\to f$ in $L^{\infty}(\omega)$, then $f_{n}\circ F\to f\circ F$
in $L^{\infty}(\omega)$ and
\[
  \|f_{n}\circ F - f\circ F\|_{\infty}
  \le \|f_{n}-f\|_{\infty}.
\]
\end{lemma}

\begin{remark}[When is $\mathcal K_F$ an isometry on $L^\infty$?]\label{rem:KoopIsoLoo}
In contrast to the case $1\le p<\infty$, the $L^\infty$-norm does not see the density $\rho_F$. The operator $\mathcal K_F$ is an isometry on $L^\infty(\omega)$ if and only if
$\omega$ and $F_{\#}\omega$ are equivalent measures (i.e.\ $\omega(A)=0$ iff $\omega(F^{-1}(A))=0$). In particular, if $F$ is measure-preserving, then $\mathcal K_F$
is an isometry on $L^\infty(\omega)$.
\end{remark}

\section{Approximate-Point Targets And \texorpdfstring{$L^1$}{L1} Upper-Bound}\label{sec:target-split}

\begin{definition}[Lower norm and three approximate-point targets]
Let $\mathcal F$ be a Banach space of observables, let $\K_F\in\mathcal B(\mathcal F)$, and set
\[
\nu_F(z):=\inf\{\|(\K_F-zI)g\|_{\mathcal F}:\|g\|_{\mathcal F}=1\},
\]
where $z\in\C$. For $\eps>0$ define the \textit{regularized approximate-point $\eps$-pseudospectrum}
\[
R_{\mathrm{ap},\eps}(\K_F) :=\overline{\{z\in\C:\nu_F(z)<\eps\}},
\]
the \textit{closed approximate-point $\eps$-pseudospectrum}
\[
C_{\mathrm{ap},\eps}(\K_F) :=\{z\in\C:\nu_F(z)\le\eps\},
\]
and the approximate point spectrum
\[
\sigap(\K_F):=\{z\in\C:\nu_F(z)=0\}.
\]
\end{definition}

\begin{remark}[Relations between the three targets]\label{rem:target-relations}
For every bounded $\K_F$ and every $\eps>0$,
\[
R_{\mathrm{ap},\eps}(\K_F)\subseteq C_{\mathrm{ap},\eps}(\K_F) \subseteq R_{\mathrm{ap},\eps+\delta}(\K_F)
\]
for $\delta>0$. Consequently,
\[
C_{\mathrm{ap},\eps}(\K_F)=\bigcap_{q=1}^{\infty} R_{\mathrm{ap},\eps+2^{-q}}(\K_F), \qquad
\sigap(\K_F)=\bigcap_{q=1}^{\infty}R_{\mathrm{ap},2^{-q}}(\K_F).
\]
These identities are the safe way to pass from strict residual tests to closed or exact targets.
\end{remark}

\begin{assumption}[Dictionary and quadrature hypotheses for the $L^1$ residual theorem]\label{ass:dictionary-quadrature-Vn}
Let $(\X,d)$ be compact, let $\omega$ be a finite Borel measure with $\omega(\X)>0$, and let $\mathcal F=L^1(\X,\omega)$.
Assume either $\supp\omega=\X$, or else every finite dictionary below is understood after quotienting $C(\X)$ into $L^1(\omega)$ and is injective in $L^1(\omega)$. Fix finite-dimensional spaces
\[
V_n=\operatorname{span}\{\psi_{n,1},\ldots,\psi_{n,d_n}\}\subset C(\X),
\]
where $n\in \mathbb{N}$, such that:
\begin{enumerate}[label=(D\arabic*),leftmargin=*]
\item $V_n\subseteq V_{n+1}$, $\One\in V_n$ for every $n$, and
$\bigcup_n V_n$ is dense in $L^1(\omega)$.
\item Writing $\Psi_n c:=\sum_{j=1}^{d_n}c_j\psi_{n,j}$, one has
$\Psi_n c\ne0$ in $L^1(\omega)$ whenever $c\ne0$.
\item There is a fixed quadrature scheme
\[
Q_m(\varphi):=\sum_{P\in\mathcal P_m}\omega(P)\varphi(x_P),\qquad \varphi\in C(\X),
\]
with mesh$(\mathcal P_m)\to0$, independent of $F$, such that $Q_m(\varphi)\to\int_\X \varphi\,d\omega$
\(\varphi\in C(\X)\), and such that the convergence is uniform on compact subsets of \(C(\X)\) in the uniform norm.
\item For every admissible continuous $F:\X \to \X$, the Koopman operator is bounded on
$L^1(\omega)$.
\end{enumerate}
\end{assumption}

\begin{definition}[Coefficient-sphere residuals]\label{def:coefficient-residual-sp}
For $z\in\C$, $F:\X \to \X$, and $n,m\in\N$, define
\[
r_{n,m}(z,F) :=
\min_{\|c\|_2=1} \frac{Q_m\bigl(|(\Psi_n c)\circ F-z\Psi_n c|\bigr)} {\|\Psi_n c\|_{L^1(\omega)}}.
\]
The denominator is independent of $F$ and is strictly positive on the compact coefficient sphere $\{c:\|c\|_2=1\}$ by \cref{ass:dictionary-quadrature-Vn}(D2).
\end{definition}

\begin{definition}[Finite grid outputs]
Let
\[
G_N:=\left\{\frac{k+i\ell}{N}:k,\ell\in\mathbb Z,\,
\left|\frac{k+i\ell}{N}\right|\le N\right\}\cup\{1\}.
\]
For \(N,m\in\mathbb N\), define
\[
r_{N,m}(z,F):=
\min_{\|c\|_2=1}
\frac{
Q_m\!\left(\left|(\Psi_Nc)\circ F-z\Psi_Nc\right|\right)
}{
\|\Psi_Nc\|_{L^1(\omega)}
}.
\]
For \(\varepsilon>0\), set
\[
\Gamma^\varepsilon_{N,m}(F)
:=
\{z\in G_N:r_{N,m}(z,F)<\varepsilon-N^{-1}\}\cup\{1\},
\]
whenever \(\varepsilon-N^{-1}>0\), and set
\[
\Gamma^\varepsilon_{N,m}(F):=\{1\}
\]
otherwise.
\end{definition}

The proof of the upper bounds in \cite[Prop.~4.3 and~4.6, Lem.~4.7, Thm.~4.8-4.9, Cor.~4.10-4.11]{sorg2025solvability} densifies to i) a nested family of finite-dimensional continuous trial spaces whose union is dense in the observable space; ii) positivity of the \(L^p\)-norm on the finite-dimensional coefficient sphere, so that the denominator in the coefficient-sphere residuals is bounded away from zero; iii) convergence of tagged quadrature residuals on compact families of continuous integrands; iv) finite-query dependence of the quadrature residuals on the point-evaluation table of \(F\) at the quadrature tags; v) the finite-grid threshold-stabilization argument used to avoid boundary oscillation at strict residual thresholds; vi) the compact-target identities
\[
        C_{\mathrm{ap},\varepsilon} = \bigcap_{q=1}^{\infty}R_{\mathrm{ap},\varepsilon+2^{-q}},\quad
        \sigma_{\mathrm{ap}}=
        \bigcap_{q=1}^{\infty} R_{\mathrm{ap},2^{-q}}.
\]
Therefore the main part is to show that this is also the case for $p=1$, which is true at least under \cref{ass:dictionary-quadrature-Vn}.

\begin{proposition}[\(L^1\) endpoint residual upper-bound]\label{prop:L1-endpoint-residual-Vn}
Assume \cref{ass:dictionary-quadrature-Vn}. Let \(\Omega\) be one of the continuous
input classes on which \(\K_F\) is bounded on \(L^1(\omega)\), and let the oracle be the
\(\Delta_1\) device consisting of point evaluations of \(F\) and the fixed quadrature data. Then the continuous-dictionary quadrature-residual construction of
\cite[Sec.~4-5]{sorg2025solvability} applies at the endpoint \(p=1\), with \(L^1(\omega)\) in place of \(L^p(\omega)\), and gives the target-split upper bounds
\[
\mathrm{SCI}_G\!\left(F\mapsto R_{\mathrm{ap},\varepsilon}(\K_F)\right)
\le
\begin{cases}
2,&\Omega\text{ has no fixed modulus of continuity},\\
1,&\Omega\text{ has a fixed modulus of continuity},
\end{cases}
\]
and
\[
\mathrm{SCI}_G\!\left(F\mapsto C_{\mathrm{ap},\varepsilon}(\K_F)\right), \quad
\mathrm{SCI}_G\!\left(F\mapsto\sigma_{\mathrm{ap}}(\K_F)\right)
\le
\begin{cases}
3,&\Omega\text{ has no fixed modulus of continuity},\\
2,&\Omega\text{ has a fixed modulus of continuity}.
\end{cases}
\]
\end{proposition}

\begin{proof}
By \cref{prop:boundednessL1}, every \(F\in\Omega\) gives a bounded Koopman operator
\[
        \K_F:L^1(\omega)\to L^1(\omega).
\]
Thus the lower norm
\[
        \nu_F(z) = \inf_{\|g\|_{L^1}=1}\|(\K_F-zI)g\|_{L^1}
\]
is well-defined.

We now check that the residual-tower proof of the \(1<p<\infty\) theory transfers to \(p=1\). Indeed, all of i)-vi) hold under \cref{ass:dictionary-quadrature-Vn} with \(L^1(\omega)\) in place of \(L^p(\omega)\). (D1) gives the nested dense dictionaries, (D2) gives
\[
        \min_{\|c\|_2=1}\|\Psi_n c\|_{L^1(\omega)}>0
\]
for every \(n\), and (D3) gives the required quadrature convergence on compact subsets of \(C(\X)\). For fixed \(n,z,F\), the family
\[
        \left\{\left|(\Psi_n c)\circ F-z\Psi_n c\right| : \|c\|_2=1\right\}\subset C(\X)
\]
is compact, because it is the continuous image of the compact coefficient sphere. Therefore
\[
        r_{n,m}(z,F)\longrightarrow r_n(z,F)
\]
as \(m\to\infty\), exactly as in \cite[Prop.~4.6]{sorg2025solvability}.

The convergence
\[
        r_n(z,F)\downarrow \nu_F(z)
\]
follows from the density of \(\bigcup_nV_n\) in \(L^1(\omega)\) and boundedness of
\(\K_F\). Explicitly, for every \(f\in L^1(\omega)\) with \(\|f\|_1=1\), choose
\(g\in V_n\) with \(g\ne0\) and \(g/\|g\|_1\to f\) in \(L^1\). Then
\[
        \|(\K_F-zI)(g/\|g\|_1)\|_1 \to \|(\K_F-zI)f\|_1.
\]
Taking infima gives \(r_n(z,F)\downarrow\nu_F(z)\). Moreover, both \(r_n(\cdot,F)\)
and \(\nu_F(\cdot)\) are \(1\)-Lipschitz, since for normalized \(g\)
\[
        \left| \|(\K_F-zI)g\|_1-\|(\K_F-wI)g\|_1 \right| \le |z-w|.
\]
Hence Dini's theorem gives uniform convergence of \(r_n(\cdot,F)\) to
\(\nu_F(\cdot)\) on compact \(z\)-sets, just as in \cite[Prop.~4.3]{sorg2025solvability}.

For fixed \(n,m\), the residual values use only finitely many values of \(F\), namely
the values of \(F\) at the quadrature tags appearing in the finitely many quadrature
rules used up to level \(m\). Hence the finite-grid residual outputs are type-\(G\)
general algorithms, as in \cite[Lem.~4.5]{sorg2025solvability}. The finite-grid
threshold-selection lemma \cite[Lem.~4.7]{sorg2025solvability} then gives an eventually
constant inner quadrature limit in the no-modulus case. Thus the proof of \cite[Thm.~4.8]{sorg2025solvability} gives
\[
\mathrm{SCI}_G \left(F\mapsto R_{\mathrm{ap},\varepsilon}(\K_F)\right)\le 2
\]
on no-modulus classes.

If \(\Omega\) has a fixed modulus of continuity, the same Arzelà-Ascoli argument as in
\cite[Prop.~4.6 and Thm.~4.9]{sorg2025solvability} gives a quadrature schedule depending
only on the dictionary, the modulus, and the grid level, not on \(F\). Hence the
regularized target is computed with one fewer limit, i.e.
\[
\mathrm{SCI}_G \left(F\mapsto R_{\mathrm{ap},\varepsilon}(\K_F)\right)\le 1.
\]

Finally, by \cref{rem:target-relations},
\[
        C_{\mathrm{ap},\varepsilon}(\K_F) = \bigcap_{q=1}^{\infty} R_{\mathrm{ap},\varepsilon+2^{-q}}(\K_F), \quad \sigma_{\mathrm{ap}}(\K_F) =
        \bigcap_{q=1}^{\infty} R_{\mathrm{ap},2^{-q}}(\K_F).
\]
For decreasing nonempty compact sets, Hausdorff convergence to the intersection follows
from compactness. Therefore passing from the regularized target to the closed
fixed-\(\varepsilon\) target or to \(\sigma_{\mathrm{ap}}\) adds one outer limit. This gives
\[
\mathrm{SCI}_G \left(F\mapsto C_{\mathrm{ap},\varepsilon}(\K_F)\right),\quad
\mathrm{SCI}_G \left(F\mapsto\sigma_{\mathrm{ap}}(\K_F)\right)
\le
\begin{cases}
3,&\Omega\text{ has no fixed modulus of continuity},\\
2,&\Omega\text{ has a fixed modulus of continuity}.
\end{cases}
\]
\end{proof}

\section{\texorpdfstring{$L^\infty$}{Loo} Approximate Point Spectrum: Block Formula And No-Go Theorem}

\begin{lemma}[Exact \(L^\infty\) block lower-norm formula]\label{lem:block-lower-norm-exact}
Let \((\X,\omega)\) be a finite measure space and let
\[
        \X= \bigsqcup_{j\in J} \X_j
\]
be a countable measurable partition with
\[
        0<\omega(\X_j)<\infty, \, j\in J.
\]
Let \(F:\X \to \X\) be measurable and nonsingular, and assume that each \(\X_j\) is \(F\)-invariant modulo null sets:
\[
        \omega\bigl(\X_j\setminus F^{-1}(\X_j)\bigr)=0
\]
for $j\in J$. Put for short \(\omega_j:=\omega|_{\X_j}\). For each \(j\), define
\[
        \K_j:=\K_{F|\X_j}:L^\infty(\X_j,\omega_j) \to L^\infty(\X_j,\omega_j)
\]
by
\[
        \K_j f := f\circ F \quad\text{on } \X_j,
\]
where the expression is understood modulo \(\omega_j\)-null sets. Then

\begin{enumerate}
\item The restriction map
\[
        \Phi:L^\infty(\X,\omega)\to \prod_{j\in J}^{\ell^\infty} L^\infty(\X_j,\omega_j), \qquad
        \Phi f := (f|_{\X_j})_{j\in J},
\]
is a well-defined linear isometric isomorphism, where
\[
        \left\|(f_j)_{j\in J}\right\|_{\ell^\infty} := \sup_{j\in J}\|f_j\|_{L^\infty(\omega_j)} .
\]

\item Under this identification, \(\K_F\) acts block-diagonally, i.e.
\[
        \Phi \K_F\Phi^{-1} = \prod_{j\in J}^{\ell^\infty} \K_j .
\]

\item If
\[
        \nu_F(\lambda) := \inf_{\|f\|_{L^\infty(\X)}=1} \|(\K_F-\lambda I)f\|_{L^\infty(\X)}
\]
and
\[
        \nu_j(\lambda) := \inf_{\|g\|_{L^\infty(\X_j)}=1} \|(\K_j-\lambda I)g\|_{L^\infty(\X_j)},
\]
then for every \(\lambda\in\mathbb C\),
\[
        \nu_F(\lambda) =
        \inf_{j\in J}\nu_j(\lambda).
\]

\item Consequently,
\[
        \sigma_{\mathrm{ap}}(\K_F) = \left\{ \lambda\in\mathbb C: \inf_{j\in J}\nu_j(\lambda)=0 \right\}.
\]
In particular,
\[
        \overline{ \bigcup_{j\in J}\sigma_{\mathrm{ap}}(\K_j)} \subseteq \sigma_{\mathrm{ap}}(\K_F).
\]
\end{enumerate}
\end{lemma}

\begin{proof}
First note that the restricted Koopman operators \(\K_j\) are well-defined. Indeed,
if \(A\subseteq \X_j\) is \(\omega_j\)-null, then \(A\) is \(\omega\)-null, and global nonsingularity gives
\[
        \omega(F^{-1}(A))=0.
\]
Hence
\[
        \omega_j\bigl(\X_j\cap F^{-1}(A)\bigr)=0.
\]
Thus composition with \(F\) preserves \(\omega_j\)-a.e. equivalence classes on the block \(\X_j\). The invariance assumption
\[
        \omega\bigl(\X_j\setminus F^{-1}(\X_j)\bigr)=0
\]
ensures that, for \(\omega_j\)-a.e. \(x\in \X_j\), the point \(F(x)\) again belongs to
\(\X_j\). Therefore \(\K_j\) is exactly the block restriction of \(\K_F\).

\noindent\textbf{(1):}
If \(f=g\) \(\omega\)-a.e. on \(\X\), then
\[
        f|_{\X_j}=g|_{\X_j} \quad \omega_j\text{-a.e.}
\]
for every \(j\). Hence the restriction map
\[
        \Phi f=(f|_{\X_j})_{j\in J}
\]
is well-defined on \(L^\infty\)-classes. Let \(f\in L^\infty(\X,\omega)\). Put
\[
        M:=\sup_{j\in J}\|f|_{\X_j}\|_{L^\infty(\omega_j)} .
\]
If \(a>\|f\|_{L^\infty(\omega)}\), then \(|f|\le a\) \(\omega\)-a.e. on \(\X\), hence
\[
        |f|\le a \quad \omega_j\text{-a.e. on } \X_j
\]
for every \(j\). Thus
\[
        M\le \|f\|_{L^\infty(\omega)}.
\]
Conversely, if \(a>M\), then for every \(j\),
\[
        |f|\le a \quad \omega_j\text{-a.e. on } \X_j.
\]
Since the partition is countable, the union of the exceptional null sets is still
\(\omega\)-null. Therefore \(|f|\le a\) \(\omega\)-a.e. on \(\X\), and so
\[
        \|f\|_{L^\infty(\omega)}\le a.
\]
Letting \(a\downarrow M\) gives
\[
        \|f\|_{L^\infty(\omega)} =
        \sup_{j\in J}\|f|_{\X_j}\|_{L^\infty(\omega_j)}.
\]
Thus \(\Phi\) is an isometry into the \(\ell^\infty\)-product.

It remains to see that \(\Phi\) is surjective. Let
\[
        (f_j)_{j\in J}\in \prod_{j\in J}^{\ell^\infty} L^\infty(\X_j,\omega_j), \, M:=\sup_j\|f_j\|_\infty<\infty.
\]
Choose measurable representatives of the \(f_j\), and define
\[
        f(x):=f_j(x) \quad\text{if } x\in \X_j.
\]
Because the partition is countable, \(f\) is measurable. Moreover,
\[
        |f|\le M \quad \omega\text{-a.e.},
\]
so \(f\in L^\infty(\X,\omega)\), and \(\Phi f=(f_j)_j\). Hence \(\Phi\) is a linear isometric isomorphism.

\noindent\textbf{(2):}
Let \(f\in L^\infty(\X,\omega)\). For fixed \(j\), the block invariance assumption gives
\[
        F(x)\in \X_j \quad\text{for }\omega_j\text{-a.e. } x\in \X_j.
\]
Hence, for \(\omega_j\)-a.e. \(x\in \X_j\),
\[
        (\K_F f)(x) = f(F(x)) = (f|_{\X_j})(F(x))= \K_j(f|_{\X_j})(x).
\]
Therefore
\[
        (\K_F f)|_{\X_j}=\K_j(f|_{\X_j}) \quad\text{in } L^\infty(\X_j,\omega_j),
\]
which proves
\[
        \Phi \K_F\Phi^{-1} = \prod_{j\in J}^{\ell^\infty} \K_j .
\]

\noindent\textbf{(3):}
Let
\[
        T:=\Phi \K_F\Phi^{-1} = \prod_{j\in J}^{\ell^\infty} \K_j .
\]
Then for \(u=(u_j)_{j\in J}\) in the \(\ell^\infty\)-product,
\[
        (T-\lambda I)u = ((\K_j-\lambda I)u_j)_{j\in J},
\]
and hence
\[
        \|(T-\lambda I)u\| =
        \sup_{j\in J} \|(\K_j-\lambda I)u_j\|.
\]

We first prove
\[
        \nu_F(\lambda)\le \inf_j \nu_j(\lambda).
\]
Fix \(j_0\in J\) and \(\eta>0\). Choose
\[
        g\in L^\infty(\X_{j_0},\omega_{j_0}), \quad \|g\|_\infty=1,
\]
such that
\[
        \|(\K_{j_0}-\lambda I)g\|_\infty \le \nu_{j_0}(\lambda)+\eta.
\]
Extend \(g\) by zero to all other blocks. In the product notation this is the vector \(u=(u_j)_j\) with
\[
        u_{j_0}=g, \quad u_j=0
\]
for $j\ne j_0$. Then \(\|u\|=1\), and
\[
        \|(T-\lambda I)u\| =
        \|(\K_{j_0}-\lambda I)g\| \le \nu_{j_0}(\lambda)+\eta.
\]
Taking the infimum over \(j_0\) and then letting \(\eta\downarrow0\) gives
\[
        \nu_F(\lambda)\le \inf_j\nu_j(\lambda).
\]

Conversely, put
\[
        \alpha:=\inf_{j\in J}\nu_j(\lambda).
\]
Let \(u=(u_j)_{j\in J}\) satisfy \(\|u\|=1\). Fix \(0<\eta<1\). Since
\[
        \sup_j\|u_j\|=1,
\]
there exists \(j_\eta\in J\) such that
\[
        \|u_{j_\eta}\|>1-\eta.
\]
Therefore
\[
\begin{aligned}
        \|(T-\lambda I)u\| &= \sup_j\|(\K_j-\lambda I)u_j\| \\
        &\ge \|(K_{j_\eta}-\lambda I)u_{j_\eta}\| \\
        &\ge \nu_{j_\eta}(\lambda)\,\|u_{j_\eta}\| \\
        &\ge \alpha(1-\eta).
\end{aligned}
\]
Letting \(\eta\downarrow0\), we obtain
\[
        \|(T-\lambda I)u\|\ge \alpha.
\]
Taking the infimum over all \(\|u\|=1\) gives
\[
        \nu_F(\lambda)\ge \inf_j\nu_j(\lambda).
\]
Thus
\[
        \nu_F(\lambda)=\inf_{j\in J}\nu_j(\lambda).
\]

\noindent\textbf{(4):}
By definition,
\[
        \lambda\in\sigma_{\mathrm{ap}}(\K_F) \quad\Longleftrightarrow\quad \nu_F(\lambda)=0.
\]
Using the lower-norm formula, this is equivalent to
\[
        \inf_{j\in J}\nu_j(\lambda)=0.
\]
Hence
\[
        \sigma_{\mathrm{ap}}(\K_F) =
        \left\{\lambda\in\mathbb C: \inf_{j\in J}\nu_j(\lambda)=0 \right\}.
\]

Finally, for each fixed \(j\),
\[
        \sigma_{\mathrm{ap}}(\K_j) =
        \{\lambda:\nu_j(\lambda)=0\} \subseteq \{\lambda:\inf_i\nu_i(\lambda)=0\} =
        \sigma_{\mathrm{ap}}(\K_F).
\]
The set \(\sigma_{\mathrm{ap}}(\K_F)\) is closed, because
\[
        \lambda\mapsto \nu_F(\lambda)
\]
is \(1\)-Lipschitz. Therefore
\[
        \overline{ \bigcup_{j\in J}\sigma_{\mathrm{ap}}(\K_j)} \subseteq \sigma_{\mathrm{ap}}(\K_F).
\]
The proof is complete.
\end{proof}

\begin{definition}[Binary trees, finite patterns, and Silver trees]\label{def:trees-patterns-silver}
Let \(2^{<\mathbb N}\) denote the set of all finite binary words.  If
\(\sigma\in 2^{<\mathbb N}\), write \(|\sigma|\) for its length and write
\[
        \sigma=(\sigma_1,\ldots,\sigma_{|\sigma|}).
\]
Let
\[
        \Tree_2 :=
        \{S\subseteq 2^{<\mathbb N}: \sigma\in S,\ \tau\preceq\sigma \Rightarrow \tau\in S\}
\]
be the space of downward closed binary trees. We regard \(\Tree_2\) as a closed
subspace of the Cantor space \(2^{2^{<\mathbb N}}\), equipped with the product
topology. Equivalently, a basic clopen condition specifies membership or non-membership
of finitely many finite words.

For \(m\in \mathbb{N}\) and
\[
        u=(u_1,\ldots,u_m)\in\{0,1,\ast\}^m,
\]
define
\[
        \Sigma(u) := \{\tau\in2^m : u_i\in\{0,1\}\Rightarrow \tau_i=u_i \text{ for all }1\le i\le m\}.
\]
Define the star-count
\[
        k(u):=\#\{1\le i\le m : u_i=\ast\}.
\]
Then
\[
        |\Sigma(u)|=2^{k(u)}.
\]
For a set \(A\subseteq\mathbb N\), we say that \(A\) is coinfinite if \(\mathbb N\setminus A\) is infinite. For a coinfinite set \(A\subseteq\mathbb N\) and a point \(x\in2^{\mathbb N}\), define the Silver tree
\[
        S_{A,x} := \{\sigma\in2^{<\mathbb N} : \sigma_i=x_i \text{ whenever }1\le i\le|\sigma| \text{ and } i\in A\}.
\]
Define the Silver-tree predicate
\[
        V := \{S\in\Tree_2 : \exists A\subseteq\mathbb N \text{ coinfinite } \exists x\in2^{\mathbb N}\text{ such that }S_{A,x}\subseteq S\}.
\]
\end{definition}

\begin{proposition}[Descriptive-set-theoretic status of \(V\) {\cite[Sec.~25, Sec.~27A]{Kec95}, \cite[Thm.~3.9]{mazurkiewicz2024idealanalyticsets}}]\label{prop:silver-nonborel}
The set \(V\subseteq\Tree_2\) is analytic and non-Borel. More precisely, the classical ill-founded-tree set Borel-reduces to \(V\).
\end{proposition}

\begin{definition}[Star-priority selected pattern]\label{def:star-priority-pattern}
Fix the priority order
\[
        \ast \prec 0 \prec 1.
\]
For \(S\in\Tree_2\), define recursively
\[
        u_m(S)\in\{0,1,\ast\}^m\cup\{\dagger\},
\]
where $m\in \mathbb{N}_0$, as follows.

Set \(u_0(S)\) to be the empty word. Suppose \(m\in \mathbb{N}\) and \(u_{m-1}(S)\) has already been defined. If
\[
        u_{m-1}(S)=\dagger,
\]
set
\[
        u_m(S):=\dagger.
\]
Otherwise, among the three extensions
\[
        u_{m-1}(S)^\frown\ast,\quad
        u_{m-1}(S)^\frown0,\quad
        u_{m-1}(S)^\frown1,
\]
choose the \(\prec\)-least word \(v\in\{0,1,\ast\}^m\) satisfying
\[
        \Sigma(v)\subseteq S\cap 2^m.
\]
If no such extension exists, set
\[
        u_m(S):=\dagger.
\]
Define
\[
        P_m(S) :=
        \begin{cases}
        \Sigma(u_m(S)),&u_m(S)\ne\dagger,\\
        \varnothing,&u_m(S)=\dagger,
        \end{cases}
\]
and
\[
        k_m(S) :=
        \begin{cases}
        k(u_m(S)),&u_m(S)\ne\dagger,\\
        0,&u_m(S)=\dagger.
        \end{cases}
\]
\end{definition}

\begin{lemma}[Finite-level dependence of \(k_m\)]\label{lem:km-finite-level-dependence}
For every \(m\in \mathbb{N}_0\), the map
\[
        S\longmapsto k_m(S)
\]
is determined by the finite initial subtree
\[
        S\cap 2^{\le m} := S\cap\bigcup_{r=0}^{m}2^r.
\]
Consequently, for every \(K\in\mathbb N\), the set
\[
        \{S\in\Tree_2 : k_m(S)\ge K\}
\]
is clopen in \(\Tree_2\).
\end{lemma}

\begin{proof}
We prove the finite-level dependence by induction on \(m\).

For \(m=0\), \(u_0(S)\) is fixed and independent of \(S\). Suppose the claim is known
up to level \(m-1\). To determine \(u_m(S)\), one first determines \(u_{m-1}(S)\),
which by the induction hypothesis depends only on \(S\cap2^{\le m-1}\). If
\(u_{m-1}(S)=\dagger\), then \(u_m(S)=\dagger\).

Otherwise, one checks the three candidate extensions
\[
        v\in
        \{u_{m-1}(S) \frown\ast,\,
          u_{m-1}(S) \frown0,\,
          u_{m-1}(S) \frown1\}.
\]
For a fixed candidate \(v\), the condition
\[
        \Sigma(v)\subseteq S\cap 2^m
\]
is a finite conjunction of coordinate conditions $\tau\in S$ for $\tau\in\Sigma(v)$. Hence it depends only on \(S\cap2^m\).  Therefore \(u_m(S)\), \(P_m(S)\), and
\(k_m(S)\) depend only on \(S\cap2^{\le m}\).

Since \(S\cap2^{\le m}\) is a finite set of coordinates in the product topology on
\(\Tree_2\), every condition determined by \(S\cap2^{\le m}\) is a finite union of basic
clopen cylinders. Thus \(\{S:k_m(S)\ge K\}\) is clopen.
\end{proof}

\begin{definition}[Finite-period tree-to-dynamics map]
\label{def:finite-period-tree-dynamics}
Let $\X := 2^{\mathbb N}$ with the standard Cantor ultrametric and Bernoulli product probability measure
\[
        \mu:=\left(\frac12\delta_0+\frac12\delta_1\right)^{\otimes\mathbb N}.
\]
Let
\[
        x_\infty:=(1,1,1,\ldots).
\]
For \(m\in \mathbb{N}\), define the clopen block
\[
        Y_m := \{x\in \X : x_1=\cdots=x_{m-1}=1,\ x_m=0\}.
\]
Then
\[
        \X = \{x_\infty\} \sqcup \bigsqcup_{m\in \mathbb{N}} Y_m.
\]
For \(\tau=(\tau_1,\ldots,\tau_m)\in2^m\), define the clopen cylinder
\[
        C_{m,\tau} := \{x\in Y_m : (x_{m+1},\ldots,x_{2m})=\tau\}.
\]
Thus
\[
        Y_m=\bigsqcup_{\tau\in2^m}C_{m,\tau}, \quad
        \mu(C_{m,\tau})=2^{-2m}.
\]

Let \(P\subseteq2^m\). If \(P\ne\varnothing\), let
\[
        \operatorname{succ}_{m,P}: P \to P
\]
be the cyclic successor map on \(P\) with respect to lexicographic order. Thus
\(\operatorname{succ}_{m,P}\) is one cycle of length \(|P|\).

On the complement \(2^m\setminus P\), define the dump involution
\[
        \pi_{m,P} : 2^m\setminus P\to 2^m\setminus P
\]
as follows. List
\[
        2^m\setminus P=\{\rho_1<_{\operatorname{lex}}\rho_2<_{\operatorname{lex}} \cdots<_{\operatorname{lex}}\rho_\ell\}.
\]
Set
\[
        \pi_{m,P}(\rho_{2a-1})=\rho_{2a},
        \qquad
        \pi_{m,P}(\rho_{2a})=\rho_{2a-1}
\]
for $1\le a\le\lfloor \ell/2\rfloor$, and, if \(\ell\) is odd, set
\[
        \pi_{m,P}(\rho_\ell)=\rho_\ell.
\]

For \(S\in\Tree_2\), define the permutation
\[
        \theta_{m,S} : 2^m \to 2^m
\]
by
\[
        \theta_{m,S}(\tau) :=
        \begin{cases}
        \operatorname{succ}_{m,P_m(S)}(\tau),
        &P_m(S)\ne\varnothing\text{ and }\tau\in P_m(S),\\
        \pi_{m,P_m(S)}(\tau),
        &\tau\notin P_m(S).
        \end{cases}
\]
The first case is absent if \(P_m(S)=\varnothing\). Define
\[
        F_S: \X \to \X, \quad F_S(x_\infty):=x_\infty,
\]
and, for \(x\in Y_m\), writing
\[
        \tau(x):=(x_{m+1},\ldots,x_{2m})\in2^m,
\]
set \(F_S(x)\) to be the unique point in \(Y_m\) satisfying
\[
        (F_S(x))_r=x_r
\]
for $1\le r\le m$ or $r>2m$, and
\[
        ((F_S(x))_{m+1},\ldots,(F_S(x))_{2m}) = \theta_{m,S}(\tau(x)).
\]
Equivalently,
\[
        F_S(C_{m,\tau})=C_{m,\theta_{m,S}(\tau)}.
\]
\end{definition}

\begin{lemma}[Approximate Point-Spectrum of the finite-period tree scheme]\label{lem:finite-period-tree-scheme-spectrum}
Let \(S\in\Tree_2\), and let \(F_S\) be the map of \cref{def:finite-period-tree-dynamics}. Let
\[
        \K_{F_S}:L^\infty(\X,\mu)\to L^\infty(\X,\mu),\quad \K_{F_S}f=f\circ F_S,
\]
be the associated Koopman operator. For \(q\in \mathbb{N}\), write
\[
        U_q:=\{\lambda\in\mathbb T : \lambda^q=1\}.
\]
Then
\[
        \sigma_{\mathrm{ap}}(\K_{F_S}) =
        \overline{ \{1,-1\}\cup \bigcup_{m\in \mathbb{N}} U_{2^{k_m(S)}} }.
\]
\end{lemma}

\begin{proof}
First observe that \(F_S\) is a measure-preserving homeomorphism. Indeed, on every
clopen block \(Y_m\), it merely permutes the equal-measure clopen cylinders $C_{m,\tau}$ with $\tau\in 2^m$ and leaves all coordinates outside \(m+1,\ldots,2m\) unchanged.  
The blocks \(Y_m\) have diameters tending to zero and accumulate only at \(x_\infty\), which is fixed.
Hence the piecewise definition glues to a homeomorphism of \(\X\), and the permutation
of equal-measure cylinders preserves \(\mu\).

We decompose \(\X \setminus \{x_\infty\}\) into positive-measure finite periodic
components. The point \(\{x_\infty\}\) is \(\mu\)-null and therefore does not affect
\(L^\infty(\X,\mu)\). Fix \(m\in \mathbb{N}\). If \(P_m(S)\ne\varnothing\), then the family of cylinders
\[
        \{C_{m,\tau}:\tau\in P_m(S)\}
\]
forms one \(F_S\)-cycle of length
\[
        |P_m(S)|=|\Sigma(u_m(S))|=2^{k_m(S)}.
\]
The complement
\[
        2^m\setminus P_m(S)
\]
is decomposed by \(\pi_{m,P_m(S)}\) into cycles of length \(2\), with possibly one
fixed point if the complement has odd cardinality. Thus every remaining positive-measure
periodic component inside \(Y_m\) has period \(1\) or \(2\).

We recall the elementary finite-cycle computation. Suppose
\[
        Z=Z_0\sqcup\cdots\sqcup Z_{L-1}
\]
modulo null sets, with \(0<\nu(Z_r)<\infty\), and suppose
\[
        F(Z_r)=Z_{r+1\!\!\!\pmod L}, \qquad F^L=\mathrm{id}
\]
modulo null sets on \(Z\). Write for short \(U:=\K_F\) on \(L^\infty(Z,\nu)\). Then \(U^L=I\). If \(\lambda^L\ne1\), then
\[
        (U-\lambda I) \left(U^{L-1}+\lambda U^{L-2}+\cdots+\lambda^{L-1}I\right) = (1-\lambda^L)I,
\]
so \(U-\lambda I\) is invertible. Hence
\[
        \lambda\notin\sigma_{\mathrm{ap}}(U).
\]
If \(\lambda^L=1\), define
\[
        f|_{Z_r}:=\lambda^{-r}.
\]
Then \(f\in L^\infty(Z,\nu)\), \(f\ne0\), and
\[
        Uf=\lambda^{-1}f.
\]
Since \(U_L\) is closed under inversion, every element of \(U_L\) is an eigenvalue.
Therefore
\[
        \sigma_{\mathrm{ap}}(U)=\sigma_p(U)=U_L.
\]

Let
\[
        A_S := \overline{ \{1,-1\}\cup \bigcup_{m\ge1}U_{2^{k_m(S)}} }.
\]
We first prove
\[
        A_S\subseteq\sigma_{\mathrm{ap}}(\K_{F_S}).
\]
The point \(1\) is an eigenvalue because \(\K_{F_S} \One = \One\). For each \(m\) with
\(P_m(S)\ne\varnothing\), the \(P_m(S)\)-cycle gives all eigenvalues in
\(U_{2^{k_m(S)}}\). If \(P_m(S)=\varnothing\), then \(k_m(S)=0\) and
\(U_{2^{k_m(S)}}=U_1=\{1\}\).

It remains only to justify \(-1\). If \(k_m(S)\in \mathbb{N}\) for some \(m\), then
\[
        -1\in U_{2^{k_m(S)}}.
\]
If \(k_m(S)=0\) for every \(m\), then consider \(m=2\). In this case
\[
        |P_2(S)|\in\{0,1\},
\]
so
\[
        |2^2\setminus P_2(S)|\in\{4,3\}.
\]
The dump involution on \(2^2\setminus P_2(S)\) therefore has at least one two-cycle.
That two-cycle gives \(-1\) as an eigenvalue on the corresponding positive-measure
periodic component. Hence \(-1\in\sigma_p(\K_{F_S})\) in all cases. Thus the set
\[
        \{1,-1\}\cup\bigcup_{m\in \mathbb{N}} U_{2^{k_m(S)}}
\]
is contained in \(\sigma_p(\K_{F_S})\), hence in \(\sigma_{\mathrm{ap}}(\K_{F_S})\).
Since \(\sigma_{\mathrm{ap}}(\K_{F_S})\) is closed, we get
\[
        A_S\subseteq\sigma_{\mathrm{ap}}(\K_{F_S}).
\]

We now prove the reverse inclusion. Let $\lambda\notin A_S$. We show that
\[
        \lambda\notin\sigma_{\mathrm{ap}}(\K_{F_S}).
\]

Let \(Z\) be any positive-measure finite periodic component in the above decomposition,
and let \(L_Z\) be its period. Then
\[
        L_Z\in \{1,2\}\cup\{2^{k_m(S)} : P_m(S)\ne\varnothing\}.
\]
Let \(U_Z\) be the Koopman operator on \(L^\infty(Z)\). We have
\[
        \nu_{\K_{F_S}}(\lambda) = \inf_Z \nu_{U_Z}(\lambda),
\]
where the infimum ranges over all positive-measure finite periodic components and
\[
        \nu_T(\lambda):= \inf_{\|f\|_\infty=1}\|(T-\lambda I)f\|_\infty.
\]
This follows directly from the isometric identification of \(L^\infty(\X,\mu)\) with the
\(\ell^\infty\)-product over the components. If \(|\lambda|\ne1\), then each \(U_Z\) is an invertible isometry, and therefore
\[
        \|(U_Z-\lambda I)f\|_\infty \ge \bigl||\lambda|-1\bigr|\,\|f\|_\infty
\]
for $f\in L^\infty(Z)$. Hence
\[
        \nu_{\K_{F_S}}(\lambda) \ge \bigl||\lambda|-1\bigr| > 0.
\]
Thus \(\lambda\notin\sigma_{\mathrm{ap}}(\K_{F_S})\).

It remains to consider \(|\lambda|=1\). If the sequence \(k_m(S)\) is unbounded, then
\[
        \overline{\bigcup_{m\in \mathbb{N}} U_{2^{k_m(S)}}} = \mathbb T.
\]
Indeed, given \(e^{2\pi i\theta}\in\mathbb T\) and \(\eta>0\), choose \(m\) with
\(2^{-k_m(S)}<\eta\), and then choose an \(p \in \mathbb{Z}\) such that
\[
        \left|\theta-\frac{p}{2^{k_m(S)}}\right| \le 2^{-k_m(S)}.
\]
Thus \(e^{2\pi ip/2^{k_m(S)}} \to e^{2\pi i\theta}\) along suitable choices. Therefore
\(A_S=\mathbb T\), contradicting \(\lambda\notin A_S\). Hence, under the present
assumption \(\lambda\notin A_S\) with \(|\lambda|=1\), the sequence \(k_m(S)\) must be
bounded.

Choose \(K\in\mathbb N_0\) such that
\[
        k_m(S)\le K
\]
for $m\in\mathbb N$. Let
\[
        \mathcal L_S := \{L_Z : Z\text{ is a positive-measure finite periodic component in the above decomposition}\}.
\]
Then \(\mathcal L_S\) is finite and
\[
        \mathcal L_S \subseteq \{1,2,2^0,2^1,\ldots,2^K\}.
\]
Because \(\lambda\notin A_S\), one has
\[
        \lambda^L\ne1
\]
for $L\in\mathcal L_S$. For a finite-cycle component \(Z\) with period \(L_Z\in\mathcal L_S\), the resolvent
formula above gives
\[
        (U_Z-\lambda I)^{-1} = (1-\lambda^{L_Z})^{-1} \left(U_Z^{L_Z-1}+\lambda U_Z^{L_Z-2}+\cdots +\lambda^{L_Z-1}I\right).
\]
Since \(|\lambda|=1\) and \(U_Z\) is an isometry,
\[
        \|(U_Z-\lambda I)^{-1}\| \le \frac{L_Z}{|1-\lambda^{L_Z}|}.
\]
Equivalently,
\[
        \nu_{U_Z}(\lambda) \ge \frac{|1-\lambda^{L_Z}|}{L_Z}.
\]
Taking the minimum over the finite set \(\mathcal L_S\) gives the existence of
\[
        \delta := \min_{L\in\mathcal L_S} \frac{|1-\lambda^L|}{L} > 0.
\]
Therefore
\[
        \nu_{\K_{F_S}}(\lambda) = \inf_Z\nu_{U_Z}(\lambda) \ge \delta > 0.
\]
Hence
\[
        \lambda\notin\sigma_{\mathrm{ap}}(\K_{F_S}).
\]
The proof is complete.
\end{proof}

\begin{proposition}[No Silver hardness from the finite-period block scheme]\label{prop:finite-period-silver-nogo-prop}
Let \(F_S\) be the finite-period tree-to-dynamics map of \cref{def:finite-period-tree-dynamics}, and let \(k_m(S)\) be defined as in \cref{def:star-priority-pattern}; in particular, \(k_m(S)=0\) when \(u_m(S)=\dagger\). Then
\[
        \sigma_{\mathrm{ap}}(\K_{F_S}) = \overline{\{1,-1\}\cup \bigcup_{m\ge1}U_{2^{k_m(S)}}}, \quad U_q:=\{\lambda\in\mathbb T:\lambda^q=1\}.
\]
Consequently, for every non-torsion point \(z_0\in\mathbb T\), that is, for every \(z_0\in\mathbb T\) satisfying
\[
        z_0^q\ne1
\]
for $q\in \mathbb{N}$, one has
\[
        z_0\in\sigma_{\mathrm{ap}}(\K_{F_S}) \quad\Longleftrightarrow\quad \sup_{m \in \mathbb{N}} k_m(S)=\infty.
\]
The pullback is therefore
\[
        \{S\in\Tree_2 : z_0\in\sigma_{\mathrm{ap}}(\K_{F_S})\} = \bigcap_{K=1}^{\infty} \bigcup_{m=1}^{\infty} \{S\in\Tree_2:k_m(S)\ge K\}.
\]
This set is Borel. Hence the finite-period construction
\[
        S\longmapsto F_S
\]
cannot realize the non-Borel Silver-tree predicate \(V\).
\end{proposition}

\begin{proof}
The spectral formula
\[
        \sigma_{\mathrm{ap}}(\K_{F_S}) =
        \overline{\{1,-1\}\cup \bigcup_{m\in \mathbb{N}} U_{2^{k_m(S)}}}
\]
follows directly from \cref{lem:finite-period-tree-scheme-spectrum}. It remains to prove the equivalence for non-torsion points and the Borelness statement.

So fix a non-torsion point $z_0\in\mathbb T$. First assume
\[
        \sup_{m\in \mathbb{N}} k_m(S)=\infty.
\]
Then
\[
        \bigcup_{m\in \mathbb{N}} U_{2^{k_m(S)}}
\]
is dense in \(\mathbb T\). Indeed, for every \(\lambda=e^{2\pi i\theta}\in\mathbb T\)
and every \(r\in\mathbb N\), choose \(m\) with \(k_m(S)\ge r\). Then some
\(2^{k_m(S)}\)-th root of unity lies within angular distance at most
\[
        2^{-k_m(S)}\le 2^{-r}
\]
from \(\lambda\). Hence
\[
        \overline{\bigcup_{m\in \mathbb{N}} U_{2^{k_m(S)}}} = \mathbb T.
\]
By the spectral formula,
\[
        \mathbb T \subseteq \sigma_{\mathrm{ap}}(\K_{F_S}),
\]
and therefore
\[
        z_0\in\sigma_{\mathrm{ap}}(\K_{F_S}).
\]

Conversely, assume
\[
        \sup_{m\in \mathbb{N}} k_m(S)<\infty.
\]
Choose \(K_0\in\mathbb N\) such that
\[
        k_m(S)\le K_0
\]
for $m\in \mathbb{N}$. Then
\[
        \{1,-1\}\cup\bigcup_{m\in \mathbb{N}} U_{2^{k_m(S)}} \subseteq \{1,-1\}\cup\bigcup_{r=0}^{K_0} U_{2^r}.
\]
The right-hand side is finite and consists only of roots of unity. Since \(z_0\) is non-torsion,
\[
        z_0\notin \{1,-1\}\cup\bigcup_{r=0}^{K_0} U_{2^r}.
\]
The finite set on the right is closed, hence the spectral formula gives
\[
        z_0\notin\sigma_{\mathrm{ap}}(\K_{F_S}).
\]
Thus
\[
        z_0\in\sigma_{\mathrm{ap}}(\K_{F_S}) \quad\Longleftrightarrow\quad \sup_{m\in \mathbb{N}} k_m(S)=\infty.
\]
The equality
\[
        \{S : z_0\in\sigma_{\mathrm{ap}}(\K_{F_S})\} = \bigcap_{K=1}^{\infty} \bigcup_{m=1}^{\infty} \{S : k_m(S)\ge K\}
\]
is now just the definition of unboundedness of the integer-valued sequence \((k_m(S))_{m\in \mathbb{N}}\).

By \cref{lem:km-finite-level-dependence}, each set
\[
        \{S : k_m(S)\ge K\}
\]
is clopen in \(\Tree_2\). Therefore
\[
        \bigcup_{m=1}^{\infty} \{S : k_m(S)\ge K\}
\]
is open, and the countable intersection over \(K\) is Borel. Hence the pullback
\[
        \{S : z_0\in\sigma_{\mathrm{ap}}(\K_{F_S})\}
\]
is Borel. By \cref{prop:silver-nonborel}, the Silver-tree predicate \(V\subseteq\Tree_2\) is
not Borel. Therefore this finite-period construction cannot satisfy
\[
        S\in V \quad\Longleftrightarrow\quad z_0\in\sigma_{\mathrm{ap}}(\K_{F_S})
\]
for any non-torsion \(z_0\in\mathbb T\). Indeed, the right-hand side defines a Borel
subset of \(\Tree_2\), while the left-hand side defines the non-Borel set \(V\).
\end{proof}

\Cref{prop:finite-period-silver-nogo-prop} does not rule out \textit{every} possible
\(L^\infty\) approximate-spectrum lower-bound strategy. It rules out the specific
finite-period block strategy: approximate eigenvalues of the \(\ell^\infty\)-product only
record the closure of the finite-period root sets, and for a fixed non-torsion point this
reduces to the Borel condition that the selected periods are unbounded.

\begin{lemma}[Two-block \(L^\infty\) lower-norm formula]\label{lem:two-block-linfty-ap}
Let \((\X,\omega)\) be a finite measure space and let
\[
        \X= \X_0\sqcup \X_1
\]
be a measurable partition modulo null sets, with
\[
        0<\omega(\X_0),\omega(\X_1)<\infty.
\]
Let \(F: \X \to \X\) be measurable and nonsingular, and assume
\[
        \omega(\X_i\setminus F^{-1}(\X_i))=0
\]
where $i=0,1$. Write for short
\[
        \K:=\K_F : L^\infty(\X,\omega) \to L^\infty(\X,\omega)
\]
and let
\[
        \K_i := \K_{F|\X_i}: L^\infty(\X_i,\omega|_{\X_i}) \to L^\infty(\X_i,\omega|_{\X_i})
\]
for $i=0,1$. Then, for every \(\lambda\in\mathbb C\),
\[
        \nu_\K (\lambda) = \min\{\nu_{\K_0}(\lambda),\nu_{\K_1}(\lambda)\}.
\]
Consequently,
\[
        \sigma_{\mathrm{ap}}(\K) = \sigma_{\mathrm{ap}}(\K_0)\cup\sigma_{\mathrm{ap}}(\K_1).
\]
\end{lemma}

\begin{proof}
The restriction map
\[
        \Phi:L^\infty(\X,\omega) \to L^\infty(\X_0,\omega|_{\X_0}) \oplus_\infty L^\infty(\X_1,\omega|_{\X_1}), \quad
        \Phi f:=(f|_{\X_0},f|_{\X_1}),
\]
is a linear isometric isomorphism. Indeed,
\[
        \|f\|_{L^\infty(\X)} = \max\{\|f|_{\X_0}\|_{L^\infty(\X_0)}, \|f|_{\X_1}\|_{L^\infty(\X_1)}\},
\]
because the partition has only two pieces and both pieces have positive measure. The invariance assumptions imply that, under this identification,
\[
        \Phi \K \Phi^{-1}= \K_0 \oplus_\infty \K_1.
\]
Thus, for \(u=(u_0,u_1)\),
\[
        (\K-\lambda I)u = ((\K_0-\lambda I)u_0,(\K_1-\lambda I)u_1)
\]
and
\[
        \|(\K-\lambda I)u\| = \max\{\|(\K_0-\lambda I)u_0\|,\|(\K_1-\lambda I)u_1\|\}.
\]

We first prove
\[
        \nu_\K(\lambda) \le \min\{\nu_{\K_0}(\lambda),\nu_{\K_1}(\lambda)\}.
\]
Fix \(i\in\{0,1\}\) and \(\eta>0\). Choose \(g_i\) with
\[
        \|g_i\|=1, \quad \|(\K_i-\lambda I)g_i\| \le \nu_{\K_i}(\lambda)+\eta.
\]
Extend \(g_i\) by zero to the other block. The resulting vector has norm \(1\), and
its residual for \(\K-\lambda I\) is exactly the residual of \(g_i\) for
\(\K_i-\lambda I\). Hence
\[
        \nu_\K(\lambda) \le \nu_{\K_i}(\lambda)+\eta.
\]
Letting \(\eta\downarrow 0\) and then minimizing over \(i\) gives the desired inequality.

Conversely, let \(u=(u_0,u_1)\) satisfy
\[
        \|u\|=\max\{\|u_0\|,\|u_1\|\}=1.
\]
Choose \(i\in\{0,1\}\) such that \(\|u_i\|=1\). Then
\[
\begin{aligned}
        \|(\K-\lambda I)u\|
        &= \max\{\|(\K_0-\lambda I)u_0\|,\|(\K_1-\lambda I)u_1\|\} \\
        &\ge \|(\K_i-\lambda I)u_i\| \\
        &\ge \nu_{\K_i}(\lambda) \\
        &\ge \min\{\nu_{\K_0}(\lambda),\nu_{\K_1}(\lambda)\}.
\end{aligned}
\]
Taking the infimum over all \(\|u\|=1\) gives
\[
        \nu_\K(\lambda) \ge \min\{\nu_{\K_0}(\lambda),\nu_{\K_1}(\lambda)\}.
\]
Therefore the lower-norm formula holds.

Finally,
\[
        \lambda\in\sigma_{\mathrm{ap}}(\K) \iff \nu_\K(\lambda)=0 \iff \min\{\nu_{\K_0}(\lambda),\nu_{\K_1}(\lambda)\}=0
\]
which is equivalent to
\[
        \lambda\in\sigma_{\mathrm{ap}}(\K_0) \quad \lor \quad \lambda\in\sigma_{\mathrm{ap}}(\K_1).
\]
This proves the spectral union formula.
\end{proof}

\begin{lemma}[A \(2\)-adic translation block has full \(L^\infty\) approximate point spectrum]\label{lem:odometer-block-full-circle}
Let $\Y:=2^{\mathbb N}$ with its standard ultrametric
\[
        d_Y(y,y') :=
        \begin{cases}
        0,&y=y',\\
        2^{-\min\{r\ge1:y_r\ne y'_r\}},&y\ne y',
        \end{cases}
\]
and Bernoulli product measure
\[
        \nu:=\left(\frac12\delta_0+\frac12\delta_1\right)^{\otimes\mathbb N}.
\]
Identify \(\Y\) with the \(2\)-adic integers \(\mathbb Z_2\) by
\[
        \iota(y):=\sum_{r=1}^\infty y_r2^{r-1}.
\]
Let
\[
        \tau: \Y \to \Y, \quad
        \tau:=\iota^{-1} \circ(x\mapsto x+2)\circ\iota .
\]
Then \(\tau\) is a measure-preserving homeomorphism of \((\Y,\nu)\), and
\[
        \sigma_{\mathrm{ap}}(\K_\tau:L^\infty(\Y,\nu)\to L^\infty(\Y,\nu)) = \mathbb T.
\]
\end{lemma}

\begin{proof}
The map \(\iota\) is the standard homeomorphism between \(2^{\mathbb N}\) and
\(\mathbb Z_2\).  ddition by \(2\) is a homeomorphism of \(\mathbb Z_2\), hence
\(\tau\) is a homeomorphism of \(\Y\).

We first prove that \(\tau\) preserves \(\nu\). For a cylinder determined by the first
\(m\) digits, say
\[
        C(a_1,\ldots,a_m) := \{y\in \Y : y_1=a_1,\ldots,y_m=a_m\},
\]
write
\[
        a:=\sum_{r=1}^m a_r2^{r-1}.
\]
Then
\[
        C(a_1,\ldots,a_m) = \{y\in \Y : \iota(y)\equiv a\pmod{2^m}\}.
\]
Therefore
\[
        \tau^{-1}C(a_1,\ldots,a_m) = \{y\in \Y : \iota(y)+2\equiv a\pmod{2^m}\} = \{y\in \Y : \iota(y)\equiv a-2\pmod{2^m}\}.
\]
The latter is again a length-\(m\) cylinder. Hence
\[
        \nu(\tau^{-1}C(a_1,\ldots,a_m)) = 2^{-m} = \nu(C(a_1,\ldots,a_m)).
\]
Since cylinders form a \(\pi\)-system generating the Borel \(\sigma\)-algebra of \(\Y\),
Dynkin's \(\pi\)-\(\lambda\) lemma gives
\[
        \nu(\tau^{-1}B)=\nu(B),
\]
for $B\in\mathcal B(\Y)$. Thus \(\K_\tau f=f\circ\tau\) is an invertible isometry on \(L^\infty(\Y,\nu)\).

We next show
\[
        \sigma_{\mathrm{ap}}(\K_\tau)\subseteq\mathbb T.
\]
If \(|\lambda|>1\), then for every \(f\in L^\infty(\Y,\nu)\),
\[
        \|(\K_\tau-\lambda I)f\|_\infty \ge |\lambda|\|f\|_\infty-\|\K_\tau f\|_\infty = (|\lambda|-1)\|f\|_\infty.
\]
If \(|\lambda|<1\), then
\[
        \|(\K_\tau-\lambda I)f\|_\infty \ge \|\K_\tau f\|_\infty-|\lambda|\|f\|_\infty = (1-|\lambda|)\|f\|_\infty.
\]
In either case, the lower norm of \(\K_\tau-\lambda I\) is positive, so
\[
        \lambda\notin\sigma_{\mathrm{ap}}(\K_\tau).
\]
Hence
\[
        \sigma_{\mathrm{ap}}(\K_\tau)\subseteq\mathbb T.
\]

It remains to prove the reverse inclusion. For \(m \in \mathbb{N}_{\ge 2}\), define the residue map
\[
        s_m: \Y \to\{0,1,\ldots,2^m-1\},\quad
        s_m(y):=\sum_{r=1}^m y_r2^{r-1}.
\]
For \(k\in\mathbb Z\), define
\[
        \chi_{m,k}(y) := \exp\left(\frac{2\pi i k}{2^m}s_m(y)\right).
\]
Then
\[
        \chi_{m,k}\in L^\infty(Y,\nu),
        \qquad
        \|\chi_{m,k}\|_\infty=1.
\]
Since \(\tau\) is addition by \(2\) in \(\mathbb Z_2\),
\[
        s_m(\tau(y))\equiv s_m(y)+2\pmod{2^m}.
\]
Therefore
\[
\begin{aligned}
        (\K_\tau\chi_{m,k})(y)
        &=\chi_{m,k}(\tau(y)) \\
        &=\exp\left(\frac{2\pi ik}{2^m}s_m(\tau(y))\right) \\
        &=\exp\left(\frac{2\pi ik}{2^m}(s_m(y)+2)\right) \\
        &=\exp\left(\frac{2\pi ik}{2^{m-1}}\right)\chi_{m,k}(y).
\end{aligned}
\]
Thus every \(2^{m-1}\)-st root of unity is a an eigenvalue of \(\K_\tau\). Hence
\[
        \bigcup_{q\in \mathbb{N}} U_{2^q} \subseteq \sigma_p(\K_\tau) \subseteq \sigma_{\mathrm{ap}}(\K_\tau), \quad
        U_{2^q}:=\{\zeta\in\mathbb T:\zeta^{2^q}=1\}.
\]
The dyadic roots of unity are dense in \(\mathbb T\). Let \(\lambda\in\mathbb T\).
Choose dyadic roots \(\lambda_q\in U_{2^q}\) with
\[
        \lambda_q \to \lambda.
\]
For each \(q\), choose an eigenfunction \(f_q\in L^\infty(\Y,\nu)\) with
\[
        \|f_q\|_\infty=1, \quad
        \K_\tau f_q=\lambda_q f_q.
\]
Then
\[
        \|(\K_\tau-\lambda I)f_q\|_\infty = |\lambda_q-\lambda|\to 0.
\]
Thus
\[
        \lambda\in\sigma_{\mathrm{ap}}(\K_\tau).
\]
Since \(\lambda\in\mathbb T\) was arbitrary, we obtain
\[
        \mathbb T\subseteq\sigma_{\mathrm{ap}}(\K_\tau).
\]
Together with the previous inclusion, this proves
\[
        \sigma_{\mathrm{ap}}(\K_\tau)=\mathbb T.
\]
\end{proof}

\begin{theorem}[Endpoint \(C^0\)-instability of \(L^\infty\) approximate point spectra]\label{thm:C0-instability}
Let $\X=2^{\mathbb N}$ with its standard ultrametric
\[
        d_\X(x,y) :=
        \begin{cases}
        0,&x=y,\\
        2^{-\min\{r\ge1:x_r\ne y_r\}},&x\ne y,
        \end{cases}
\]
and Bernoulli product probability measure
\[
        \mu:=\left(\frac12\delta_0+\frac12\delta_1\right)^{\otimes\mathbb N}.
\]
There are measure-preserving homeomorphisms
\[
        F^{(n)}: \X \to \X
\]
such that
\[
        F^{(n)} \to \mathrm{id}_\X \qquad\text{uniformly}
\]
and
\[
        \sigma_{\mathrm{ap}}(\K_{\mathrm{id}_\X}: L^\infty(\X,\mu) \to L^\infty(\X,\mu)) = \{1\},
\]
whereas
\[
        \sigma_{\mathrm{ap}}(\K_{F^{(n)}}: L^\infty(\X,\mu) \to L^\infty(\X,\mu)) = \mathbb T
\]
for $n \in \mathbb{N}$. Consequently,
\[
        d_H \left(\sigma_{\mathrm{ap}}(\K_{\mathrm{id}_\X}),\sigma_{\mathrm{ap}}(\K_{F^{(n)}})\right)=2
\]
for $n\in \mathbb{N}$.
\end{theorem}

\begin{proof}
Let
\[
        x_\infty:=(1,1,1,\ldots)\in \X.
\]
For \(n\in \mathbb{N}\), define the clopen cylinder
\[
        U_n := \{x\in \X : x_1=\cdots=x_{n-1}=1,\ x_n=0\}.
\]
Then
\[
        \X = \{x_\infty\}\sqcup\bigsqcup_{n\in \mathbb{N}} U_n,
\]
and
\[
        \mu(U_n)=2^{-n}.
\]
Moreover, every two points in \(U_n\) agree in their first \(n\) coordinates, hence
\[
        \operatorname{diam}(U_n)\le 2^{-(n+1)}.
\]
Let
\[
        \Y:=2^{\mathbb N}
\]
with Bernoulli product measure
\[
        \nu:=\left(\frac12\delta_0+\frac12\delta_1\right)^{\otimes\mathbb N}.
\]
Let
\[
        \tau: \Y \to \Y
\]
be the \(2\)-adic translation by \(2\) from \cref{lem:odometer-block-full-circle}. Thus \(\tau\) is a measure-preserving homeomorphism and
\[
        \sigma_{\mathrm{ap}}(\K_\tau: L^\infty(\Y,\nu)\to L^\infty(\Y,\nu))= \mathbb T.
\]
For each \(n\in \mathbb{N}\), define the tail homeomorphism
\[
        \pi_n : U_n\to \Y,\quad
        \pi_n(1,\ldots,1,0,t_1,t_2,\ldots):=(t_1,t_2,\ldots).
\]
Equivalently,
\[
        \pi_n(1^{n-1}0 t)=t.
\]
Now define
\[
        F^{(n)}: \X \to \X
\]
by
\[
        F^{(n)}(x) :=
        \begin{cases}
        x,&x\in X\setminus U_n,\\[1mm]
        \pi_n^{-1}\bigl(\tau(\pi_n(x))\bigr),&x\in U_n.
        \end{cases}
\]
We prove the claimed properties.

\medskip

\noindent\textbf{Step 1: \(F^{(n)}\) is a measure-preserving homeomorphism:}
The sets \(U_n\) and \(\X \setminus U_n\) are clopen. On \(\X\setminus U_n\),
\(F^{(n)}\) is the identity. On \(U_n\), it is the conjugate
\[
        \pi_n^{-1}\circ\tau\circ\pi_n
\]
of the homeomorphism \(\tau\). Since the two pieces are clopen, the piecewise map
\(F^{(n)}\) is a homeomorphism of \(\X\).

We now check measure preservation. For every Borel set \(A\subseteq \Y\),
\[
        \mu(\pi_n^{-1}(A))=2^{-n}\nu(A).
\]
This follows first for cylinders and then for all Borel sets by Dynkin's \(\pi\)-\(\lambda\) lemma.

Let \(B\subseteq \X\) be Borel. Write
\[
        B_0:=B\cap(\X \setminus U_n),\quad
        B_1:=B\cap U_n.
\]
Since \(F^{(n)}\) is the identity on \(\X \setminus U_n\) and maps \(U_n\) onto itself,
\[
        (F^{(n)})^{-1}(B) = B_0 \sqcup \pi_n^{-1}\bigl(\tau^{-1}(\pi_n(B_1))\bigr).
\]
Therefore, using the relation between \(\mu|_{U_n}\) and \(\nu\), and \(\tau_\#\nu=\nu\) we get
\[
\begin{aligned}
        \mu((F^{(n)})^{-1}(B))
        &=\mu(B_0)+\mu \left(\pi_n^{-1}\bigl(\tau^{-1}(\pi_n(B_1))\bigr)\right) \\
        &=\mu(B_0)+2^{-n}\nu \left(\tau^{-1}(\pi_n(B_1))\right) \\
        &=\mu(B_0)+2^{-n}\nu(\pi_n(B_1)) \\
        &=\mu(B_0)+\mu(B_1) \\
        &=\mu(B).
\end{aligned}
\]
Thus
\[
        (F^{(n)})_\#\mu=\mu.
\]

\medskip

\noindent\textbf{Step 2: \(F^{(n)}\to\mathrm{id}_\X\) uniformly:}
If \(x\notin U_n\), then
\[
        F^{(n)}(x)=x.
\]
If \(x\in U_n\), then both \(x\) and \(F^{(n)}(x)\) belong to \(U_n\). Since all points
of \(U_n\) agree in their first \(n\) coordinates,
\[
        d_\X (F^{(n)}(x),x)\le \operatorname{diam}(U_n)\le 2^{-(n+1)}.
\]
Hence
\[
        \sup_{x\in X}d_\X (F^{(n)}(x),x)\le 2^{-(n+1)} \longrightarrow 0.
\]
Therefore
\[
        F^{(n)}\to\mathrm{id}_\X
\]
uniformly.

\medskip

\noindent\textbf{Step 3: the identity has approximate point spectrum \(\{1\}\):}
This is very elementary to see: For the identity map,
\[
        \K_{\mathrm{id}_\X}=I
\]
on \(L^\infty(\X,\mu)\). Hence
\[
        1\in\sigma_{\mathrm{ap}}(\K_{\mathrm{id}_\X}).
\]
If \(\lambda\ne1\), then for every \(f\in L^\infty(\X,\mu)\),
\[
        \|(I-\lambda I)f\|_\infty = |1-\lambda|\,\|f\|_\infty.
\]
Thus the lower norm of \(I-\lambda I\) is \(|1-\lambda|>0\), and so
\[
        \lambda\notin\sigma_{\mathrm{ap}}(\K_{\mathrm{id}_\X}).
\]
Therefore
\[
        \sigma_{\mathrm{ap}}(\K_{\mathrm{id}_\X})=\{1\}.
\]

\medskip

\noindent\textbf{Step 4: each perturbation has full-circle approximate point spectrum:}
For fixed \(n\), consider the two-block partition
\[
        \X= (\X \setminus U_n)\sqcup U_n.
\]
Both blocks have positive measure and are invariant under \(F^{(n)}\). On
\(\X \setminus U_n\), the map \(F^{(n)}\) is the identity. On \(U_n\), the map is
conjugate to \(\tau\) by \(\pi_n\).

By \cref{lem:two-block-linfty-ap},
\[
        \sigma_{\mathrm{ap}}(\K_{F^{(n)}}) = \sigma_{\mathrm{ap}}(\K_{F^{(n)}|\X \setminus U_n}) \cup \sigma_{\mathrm{ap}}(\K_{F^{(n)}|U_n}).
\]
The first block is the identity block, so
\[
        \sigma_{\mathrm{ap}}(\K_{F^{(n)}|\X \setminus U_n}) = \{1\}.
\]

For the second block, define
\[
        J_n:L^\infty(\Y,\nu) \to L^\infty(U_n,\mu|_{U_n}),\quad J_nh:=h\circ\pi_n.
\]
Because
\[
        \mu(\pi_n^{-1}(A))=2^{-n}\nu(A),
\]
the map \(J_n\) is an isometric isomorphism of \(L^\infty\)-spaces. Moreover,
\[
        \K_{F^{(n)}|U_n}J_n = J_n \K_\tau.
\]
Indeed, for \(x\in U_n\),
\[
\begin{aligned}
        (\K_{F^{(n)}|U_n}J_nh)(x)
        &=(J_nh)(F^{(n)}(x)) \\
        &=h(\pi_n(F^{(n)}(x))) \\
        &=h(\tau(\pi_n(x))) \\
        &=(\K_\tau h)(\pi_n(x)) \\
        &=(J_n \K_\tau h)(x).
\end{aligned}
\]
Thus \(\K_{F^{(n)}|U_n}\) and \(\K_\tau\) are isometrically conjugate. Therefore
\[
        \sigma_{\mathrm{ap}}(\K_{F^{(n)}|U_n}) = \sigma_{\mathrm{ap}}(\K_\tau) = \mathbb T
\]
by \cref{lem:odometer-block-full-circle}. Consequently,
\[
        \sigma_{\mathrm{ap}}(\K_{F^{(n)}}) = \{1\}\cup\mathbb T = \mathbb T.
\]

\medskip

\noindent\textbf{Step 5: the Hausdorff distance is \(2\):}
Since \(1\in\mathbb T\),
\[
        \sup_{a\in\{1\}}\operatorname{dist}(a,\mathbb T)=0.
\]
Also,
\[
        \sup_{\zeta\in\mathbb T}\operatorname{dist}(\zeta,\{1\}) = \sup_{|\zeta|=1}|\zeta-1| = 2,
\]
with the maximum attained at \(\zeta=-1\). Therefore
\[
        d_H(\{1\},\mathbb T)=2.
\]
By Steps~3 and 4,
\[
        d_H \left(\sigma_{\mathrm{ap}}(\K_{\mathrm{id}_\X}),\sigma_{\mathrm{ap}}(\K_{F^{(n)}})\right)=d_H(\{1\},\mathbb T)=2.
\]
This completes the proof.
\end{proof}

\section{\texorpdfstring{$L^\infty$}{Loo} Exact Point Spectrum: No Analytic Fixed-Eigenvalue Route}

\begin{proposition}[\(L^\infty\) and \(L^2\) fixed eigenvalues coincide in the measure-preserving case]\label{prop:Linf-L2-point-spec}
Let \((\X,\omega)\) be a finite measure space and let $F: \X \to \X$ be measurable and measure-preserving. Then, for every \(\lambda\in\mathbb C\),
\[
\lambda\in \sigma_{\mathrm p} \left(\K_F : L^\infty(\X,\omega) \to L^\infty(\X,\omega)\right) \quad\Longleftrightarrow\quad \lambda\in \sigma_{\mathrm p} \left(\K_F : L^2(\X,\omega) \to L^2(\X,\omega)\right).
\]
\end{proposition}

\begin{proof}
If \(\omega(\X)=0\), then both \(L^\infty(\X,\omega)\) and \(L^2(\X,\omega)\) are the zero space, so both point spectra are empty. Hence the assertion is trivial. We assume
from now on that
\[
        0<\omega(\X)<\infty .
\]

Since \(F_\#\omega=\omega\), the Koopman operator is an isometry on every
\(L^p(\X,\omega)\), \(1\le p\le\infty\) by \cref{prop:boundednessL1} and \cref{rem:KoopIsoLoo}.

First assume
\[
        \lambda\in \sigma_{\mathrm p} \left(\K_F : L^\infty \to L^\infty\right).
\]
Then there exists
\[
        0\ne g\in L^\infty(\X,\omega)
\]
such that
\[
        g\circ F=\lambda g \qquad\omega\text{-a.e.}
\]
Since \(\omega(\X)<\infty\), we have
\[
        L^\infty(\X,\omega)\subseteq L^2(\X,\omega).
\]
Moreover, \(g\ne 0\) in \(L^\infty\) implies \(g\ne 0\) in \(L^2\). Therefore the same
equation shows that \(g\) is a nonzero \(L^2\)-eigenfunction with eigenvalue \(\lambda\).
Thus
\[
        \lambda\in \sigma_{\mathrm p} \left(\K_F : L^2 \to L^2\right).
\]

Conversely, assume
\[
        \lambda\in \sigma_{\mathrm p} \left(\K_F : L^2\to L^2\right).
\]
Then there exists
\[
        0\ne f\in L^2(\X,\omega)
\]
such that
\[
        f\circ F=\lambda f \qquad\omega\text{-a.e.}
\]
Because \(\K_F\) is an isometry on \(L^2\), we have
\[
        \|f\|_2 = \|f\circ F\|_2 = |\lambda|\,\|f\|_2.
\]
Since \(f\ne 0\), this implies $|\lambda|=1$.

We now construct a bounded eigenfunction from \(f\). Since \(f\ne 0\) in \(L^2\),
\[
        \omega(\{|f|>0\})>0.
\]
Also \(f\) is finite \(\omega\)-a.e. Hence
\[
        \{|f|>0\} = \bigcup_{n=1}^{\infty} \left\{\frac1n \le |f|\le n\right\} \quad\text{modulo null sets}.
\]
Therefore there exists \(n\in\mathbb N\) such that
\[
        E:=\left\{\frac1n\le |f|\le n\right\}
\]
has positive measure. Put
\[
        a:=\frac1n, \quad b:=n.
\]
Thus
\[
        0<a<b<\infty, \quad \omega(E)>0, \quad E=\{a\le |f|\le b\}.
\]
From
\[
        f\circ F=\lambda f \quad \land \quad |\lambda|=1,
\]
we get
\[
        |f|\circ F=|f| \qquad\omega\text{-a.e.}
\]
Consequently,
\[
        \mathbf 1_E\circ F= \mathbf 1_E \qquad\omega\text{-a.e.}
\]
Indeed, outside a null set,
\[
        F(x)\in E \quad\Longleftrightarrow\quad a\le |f(F(x))|\le b \quad\Longleftrightarrow\quad a\le |f(x)|\le b \quad\Longleftrightarrow\quad x\in E.
\]
Now define
\[
        g:=f\,\mathbf 1_E.
\]
Then
\[
        |g|\le b \qquad\omega\text{-a.e.},
\]
so
\[
        g\in L^\infty(\X,\omega).
\]
Moreover, \(g\ne 0\), because
\[
        |g|\ge a \quad\text{on }E
\]
and \(\omega(E)>0\). Finally,
\[
\begin{aligned}
        g\circ F
        &=(f\,\mathbf 1_E)\circ F \\
        &=(f\circ F)(\mathbf 1_E\circ F) \\
        &=\lambda f\,\mathbf 1_E \\
        &=\lambda g \qquad\omega\text{-a.e.}
\end{aligned}
\]
Thus \(g\) is a nonzero \(L^\infty\)-eigenfunction with eigenvalue \(\lambda\). Hence
\[
        \lambda\in \sigma_{\mathrm p} \left(\K_F:L^\infty \to L^\infty\right).
\]
The proof is complete.
\end{proof}

The previous section ruled out one approximate-spectrum route to analytic hardness.
The next two propositions show that the analogous fixed-eigenvalue point-spectrum route
is also unavailable in the measure-preserving class: for a fixed label \(\lambda\), the
membership problem is Borel. The later calibration theorem therefore uses a sequence of
labels and finite Borel quantifier complexity, not one analytic fixed-label predicate.

\begin{proposition}[Borelness of fixed point-eigenvalue membership]\label{prop:point-borel-fixed-mem}
Let \(\X\) be a compact metric space, let \(\omega\) be a finite Borel measure on \(\X\),
and equip \(C(\X,\X)\) with the uniform metric. Fix $\lambda\in\mathbb C$. Then the set
\[
\mathcal E_\lambda := \left\{F\in C(\X,\X): F_\#\omega=\omega \text{ and }\lambda\in \sigma_{\mathrm p} \left(\K_F : L^\infty(\X,\omega) \to L^\infty(\X,\omega)\right)\right\}
\]
is Borel in \(C(\X,\X)\).

Consequently, if \(\Y\) is a standard Borel space, \(A\subseteq \Y\) is non-Borel, and
\[
        \Phi: \Y \to C(\X,\X)
\]
is Borel, then one cannot have
\[
        y\in A \quad\Longleftrightarrow\quad \Phi(y)\in\mathcal E_\lambda
\]
for $y\in \Y$. In particular, no Borel coding of a non-Borel tree predicate, such as the Silver-tree
predicate, into the measure-preserving continuous class can reduce that predicate to one
fixed \(L^\infty\) point-eigenvalue membership question.
\end{proposition}

\begin{proof}
If \(\omega(\X)=0\), then all \(L^p(\X,\omega)\)-spaces are the zero space and the point
spectrum is empty. Hence
\[
        \mathcal E_\lambda=\varnothing,
\]
which is Borel. We therefore assume
\[
        0<\omega(\X)<\infty .
\]
We first note that the measure-preserving condition is Borel. Since \(\X\) is compact
metric, \(C(\X)\) is separable in the uniform norm. Fix a countable uniformly dense set
\[
        (\varphi_\ell)_{\ell\ge1}\subset C(\X).
\]
For each \(\ell\), the map
\[
        C(\X,\X)\to\mathbb R,\quad
        F\mapsto \int_\X \varphi_\ell(F(x))\,d\omega(x)
\]
is continuous with respect to the uniform metric. Indeed, if \(F_n\to F\) uniformly,
then by uniform continuity of \(\varphi_\ell\),
\[
        \varphi_\ell\circ F_n\to \varphi_\ell\circ F
\]
uniformly, hence also in \(L^1(\omega)\). Therefore
\[
        \int_\X \varphi_\ell(F_n(x))\,d\omega(x) \to
        \int_\X \varphi_\ell(F(x))\,d\omega(x).
\]
Moreover recall,
\[
        F_\#\omega=\omega
\]
if and only if
\[
        \int_\X \varphi_\ell(F(x))\,d\omega(x) =
        \int_\X \varphi_\ell(x)\,d\omega(x)
\]
for $\ell \in \mathbb{N}$. Thus
\[
        \mathrm{MP}:=\{F\in C(\X,\X) : F_\#\omega=\omega\}
\]
is a countable intersection of closed sets, hence Borel.

If \(|\lambda|\ne1\), then
\[
        \mathcal E_\lambda=\varnothing.
\]
Indeed, for \(F\in\mathrm{MP}\), the Koopman operator is an isometry on
\(L^\infty(\X,\omega)\). If
\[
        K_F g=\lambda g
\]
for some nonzero \(g\in L^\infty(\X,\omega)\), then
\[
        \|g\|_\infty = \|\K_Fg\|_\infty = |\lambda|\,\|g\|_\infty,
\]
and therefore \(|\lambda|=1\). Hence the claim is trivial for \(|\lambda|\ne1\).

Assume now that $|\lambda|=1$. By \cref{prop:Linf-L2-point-spec}, for every \(F\in\mathrm{MP}\),
\[
\lambda\in \sigma_{\mathrm p} \left(\K_F : L^\infty \to L^\infty\right) \quad\Longleftrightarrow\quad \lambda\in \sigma_{\mathrm p} \left(\K_F : L^2 \to L^2\right).
\]
It is therefore enough to express the \(L^2\)-eigenvalue condition in a Borel way.

Choose a sequence
\[
        (h_j)_{j\in \mathbb{N}} \subset C(\X,\mathbb C)
\]
which is dense in the \(L^2(\omega)\)-unit sphere
\[
        S_{L^2}:=\{h\in L^2(\X,\omega) : \|h\|_2=1\}.
\]
Such a sequence exists because \(C(\X,\mathbb C)\) is dense in \(L^2(\X,\omega)\) and
\(L^2(\X,\omega)\) is separable.

For \(F\in C(\X,\X)\), \(N\in\mathbb N\), and \(h\in C(\X,\mathbb C)\), define
\[
        A_N^\lambda(F)h := \frac1N\sum_{r=0}^{N-1} \overline{\lambda}^{\,r}\,h\circ F^r,
\]
where \(F^0:=\mathrm{id}_\X\) and \(F^{r+1}:=F\circ F^r\). For fixed \(N\) and \(h\), the function
\[
        A_N^\lambda(F)h
\]
is continuous on \(\X\), hence belongs to \(L^2(\X,\omega)\).

For fixed \(N\) and \(j\), the map
\[
        F\mapsto \|A_N^\lambda(F)h_j\|_{L^2(\omega)}
\]
is continuous on \(C(\X,\X)\). To see this, suppose \(F_n\to F\) uniformly. For each
fixed \(r\), one has
\[
        F_n^r\to F^r
\]
uniformly, by induction on \(r\) and uniform continuity of \(F\) on the compact space
\(\X\). Since \(h_j\) is uniformly continuous,
\[
        h_j\circ F_n^r\to h_j\circ F^r
\]
uniformly. Hence
\[
        A_N^\lambda(F_n)h_j \to A_N^\lambda(F)h_j
\]
uniformly, and therefore in \(L^2(\omega)\).

Now fix \(F\in\mathrm{MP}\). Then
\[
        \K_F : L^2(\X,\omega) \to L^2(\X,\omega)
\]
is an isometry. Put
\[
        U_{\lambda,F}:=\overline{\lambda}\,\K_F .
\]
This is again an isometry on \(L^2(\X,\omega)\). The von Neumann mean ergodic theorem, see, for example \cite[Thm.~8.6]{eisner2015operator}, 
applied to \(U_{\lambda,F}\) gives strong convergence
\[
        \frac1N\sum_{r=0}^{N-1}U_{\lambda,F}^r \longrightarrow P_{\lambda,F},
\]
where \(P_{\lambda,F}\) is the mean-ergodic projection onto
\[
        \operatorname{Fix}(U_{\lambda,F}) = \{h\in L^2 : U_{\lambda,F}h=h\} = \ker(\K_F-\lambda I).
\]
Equivalently,
\[
        A_N^\lambda(F)h \longrightarrow P_{\lambda,F}h \quad\text{in }L^2
\]
for every \(h\in L^2(\X,\omega)\).

We claim that, for \(F\in\mathrm{MP}\),
\[
\lambda\in \sigma_{\mathrm p} \left(\K_F : L^2 \to L^2\right)
\]
if and only if
\[
\exists j\in\mathbb N\ \exists q\in\mathbb N\ \forall N_0\in\mathbb N\ \exists N\ge N_0: \|A_N^\lambda(F)h_j\|_{L^2}>q^{-1}. \tag{\(\ast\)}
\label{eq:point-borel-star}
\]
First assume that \(\lambda\) is an \(L^2\)-eigenvalue. Then
\[
        \ker(\K_F-\lambda I)\ne\{0\},
\]
so
\[
        P_{\lambda,F}\ne 0.
\]
Choose \(h\in S_{L^2}\) with
\[
        \|P_{\lambda,F}h\|_2>0.
\]
Since \(P_{\lambda,F}\) is bounded and \((h_j)\) is dense in the unit sphere, there
exists a \(j\) such that
\[
        \|P_{\lambda,F}h_j\|_2>0.
\]
Choose a \(q\in\mathbb N\) with
\[
        q^{-1}<\frac12\|P_{\lambda,F}h_j\|_2.
\]
Because
\[
        A_N^\lambda(F)h_j\to P_{\lambda,F}h_j \qquad\text{in }L^2,
\]
there exists \(N_1\) such that
\[
        \|A_N^\lambda(F)h_j\|_2>q^{-1}
\]
for all $N\ge N_1$. Hence \eqref{eq:point-borel-star} holds.

Conversely, assume \eqref{eq:point-borel-star}. For the corresponding \(j,q\), the
sequence
\[
        \|A_N^\lambda(F)h_j\|_2
\]
converges to
\[
        \|P_{\lambda,F}h_j\|_2.
\]
Since every tail contains an \(N\) for which the norm is \(>q^{-1}\), the limit must
satisfy
\[
        \|P_{\lambda,F}h_j\|_2\ge q^{-1}>0.
\]
Thus \(P_{\lambda,F}\ne0\), and therefore
\[
        \ker(\K_F-\lambda I)\ne\{0\}.
\]
Hence \(\lambda\) is an \(L^2\)-eigenvalue.

Combining the preceding equivalence with \cref{prop:Linf-L2-point-spec}, we obtain
\[
\mathcal E_\lambda = \mathrm{MP} \cap \bigcup_{j=1}^{\infty} \bigcup_{q=1}^{\infty} \bigcap_{N_0=1}^{\infty} \bigcup_{N\ge N_0} \left\{F\in C(\X,\X): \|A_N^\lambda(F)h_j\|_{L^2}>q^{-1} \right\}.
\]
The set \(\mathrm{MP}\) is Borel, and each set
\[
\left\{F : \|A_N^\lambda(F)h_j\|_{L^2}>q^{-1} \right\}
\]
is open by the continuity proved above. Therefore \(\mathcal E_\lambda\) is Borel.

Finally, let \(\Y\) be a standard Borel space, let \(A\subseteq \Y\) be non-Borel, and let
\[
        \Phi: \Y \to C(\X,\X)
\]
be Borel. Since \(\mathcal E_\lambda\) is Borel, the preimage $\Phi^{-1}(\mathcal E_\lambda)$ is Borel in \(\Y\). Therefore it cannot equal the non-Borel set \(A\). This proves the stated no-reduction consequence.
\end{proof}

\begin{remark}[Why this does not contradict the point-spectral calibration theorem]\label{rem:point-borel-no-contradiction}
\Cref{prop:point-borel-fixed-mem} rules out only one specific route: a Borel coding of a non-Borel predicate into membership of one fixed eigenvalue
\[
        \lambda\in\sigma_{\mathrm p}(\K_F : L^\infty\to L^\infty).
\]
The finite-height calibration below is different. It uses a sequence of labelled eigenvalues $(\lambda_j)_{j\in \mathbb{N}}$ and encodes finite Borel complexity by applying alternating quantifiers over the labels. Thus the complexity produced by the calibration family is finite Borel/tower-height complexity, not analytic non-Borel complexity for a single fixed eigenvalue.
\end{remark}

\section{Point-Spectral Koopman Calibration Families}\label{sec:point-spectrum-calibration}
The previous section shows that one fixed eigenvalue cannot encode analytic non-Borel hardness in the measure-preserving continuous class. We now use a different
mechanism. Instead of one fixed eigenvalue, we use a sequence of labelled roots of unity. Each label records one bit of a Cantor-matrix input, and alternating quantifiers
over the labels recover the finite Borel-height source problems from \cite[Sec.~9]{sorg2026foundational}. We recall for that first the Cantor-setting we worked in several times in this work before.

\begin{definition}[Cantor space, blocks, and point-evaluation oracle]
Let
\[
\X:=2^\N, \qquad d(x,y):=\begin{cases}
0,&x=y,\\
2^{-\min\{r\ge1:x_r\ne y_r\}},&x\ne y,
\end{cases}
\]
and let $\mu=(\tfrac12\delta_0+\tfrac12\delta_1)^{\otimes\N}$ be the Bernoulli product probability measure. Let
\[
x_\infty:=(1,1,1,\ldots),\quad
\Y_j:=\{x\in \X : x_1=\cdots=x_{j-1}=1,\\ x_j=0\},
\]
where $j\in \mathbb{N}$. Then $\X=\{x_\infty\}\sqcup\bigsqcup_{j\in \mathbb{N}} \Y_j$ and $\diam(\Y_j)\to 0$. Fix once and for all the injective coding
\[
        \iota_\X : \X\to\mathbb C,\quad
        \iota_\X(x):=\sum_{r=1}^{\infty}2x_r3^{-r}.
\]
The point-evaluation oracle is
\[
        \Lambda_{\mathrm{pt}} := \{\operatorname{ev}_x : x\in \X\},\quad
        \operatorname{ev}_x(F):=\iota_\X(F(x)).
\]
\end{definition}

\begin{definition}[Prime labels]
Fix a sequence $(q_j)_{j\in \mathbb{N}}$ of pairwise distinct odd primes and set
\[
\lambda_j:=\exp(2\pi i/q_j),\quad
U_{q_j}:=\{\lambda\in\T : \lambda^{q_j}=1\}.
\]
\end{definition}

\begin{definition}[Equal-measure prime-cycle cylinders]
For each $j\in \mathbb{N}$, choose an integer $N_j>j$ such that
\[
2^{N_j-j}\ge q_j.
\]
Choose pairwise distinct words
\[
\alpha_{j,0},\alpha_{j,1},\ldots,\alpha_{j,q_j-1}\in 2^{N_j}
\]
all extending the word $1^{j-1}0$. Let
\[
C_{j,r}:=[\alpha_{j,r}]:=\{x\in \X : x|_{\{1,\ldots,N_j\}}=\alpha_{j,r}\}, \quad 0\le r<q_j.
\]
Then $C_{j,r}\subset \Y_j$ are pairwise disjoint clopen cylinders and $\mu(C_{j,r})=2^{-N_j}$ for every $r$.
For two words $\alpha,\beta\in 2^N$ of the same length, write
\[
H_{\alpha,\beta}: [\alpha]\to[\beta], \quad
H_{\alpha,\beta}(\alpha\concat t):=\beta\concat t.
\]
This is a measure-preserving homeomorphism between the two cylinders.
Finally set the marker point
\[
\xi_j:=\alpha_{j,0}\concat(0,0,0,\ldots)\in C_{j,0}.
\]
\end{definition}

\begin{definition}[Local prime-cycle homeomorphisms]
For $b\in\{0,1\}$ define $\Phi_{j,b}:\Y_j\to \Y_j$ by
\[
\Phi_{j,0}:=\id_{\Y_j},
\]
and
\[
\Phi_{j,1}(x):=
\begin{cases}
H_{\alpha_{j,r},\alpha_{j,r+1\, (\mathrm{mod}\ q_j)}}(x),&x\in C_{j,r},\quad 0\le r<q_j,\\
x,&x\in \Y_j\setminus\displaystyle\bigcup_{r=0}^{q_j-1}C_{j,r}.
\end{cases}
\]
Thus $\Phi_{j,1}$ consists of one $q_j$-cycle of equal-measure clopen cylinders and the identity on the remaining clopen part of $\Y_j$.
\end{definition}

\begin{definition}[Matrix-to-Koopman coding]\label{def:matrix-to-koopman-coding}
Let
\[
\Omega_1:=\{0,1\}^{\N\times\N}
\]
with its product topology and coordinate oracle
$\Lambda_{\mathrm{mat}}:=\{\ev_{a,b}: A\mapsto A(a,b):(a,b)\in\N^2\}$.
Fix a bijection
\[
\eta:\N\to\N^2,\quad \ell:=\eta^{-1}:\N^2\to\N.
\]
For $A\in\Omega_1$ define the bit
\[
b_j(A):=A(\eta(j))
\]
for $j \in \mathbb{N}$. Define then $F_A: \X \to \X$ by
\[
F_A(x_\infty):=x_\infty,\quad F_A|_{\Y_j}:=\Phi_{j,b_j(A)}
\]
for $j\in \mathbb{N}$. Set
\[
\K_{\mathrm{pt}}:=\{F_A : A\in\Omega_1\}\subset C(\X,\X).
\]
\end{definition}

\begin{proposition}[Regularity of the coding family]\label{prop:coding-regularity-family}
For every $A\in\Omega_1$, the map $F_A$ is a measure-preserving homeomorphism of
$(\X,\mu)$. Moreover, the coding map
\[
\mathcal C:\Omega_1\to C(\X,\X),\qquad \mathcal C(A):=F_A,
\]
is continuous.
\end{proposition}

\begin{proof}
For fixed \(j\), the map \(\Phi_{j,0}\) is the identity on \(\Y_j\), and
\(\Phi_{j,1}\) is obtained by cyclically permuting finitely many equal-measure clopen cylinders
\[
        C_{j,0},\ldots,C_{j,q_j-1}
\]
and acting as the identity on the clopen complement
\[
        \Y_j\setminus\bigcup_{r=0}^{q_j-1}C_{j,r}.
\]
Thus each \(\Phi_{j,b}\) is a measure-preserving homeomorphism of \(\Y_j\). The sets \(\Y_j\) are pairwise disjoint clopen subsets of \(\X\), and
\[
        \X=\{x_\infty\}\sqcup\bigsqcup_{j\in \mathbb{N}} \Y_j.
\]
On each \(\Y_j\), the map \(F_A\) is the homeomorphism \(\Phi_{j,b_j(A)}\), and
\(F_A(x_\infty)=x_\infty\). Since
\[
        \operatorname{diam}(\Y_j)\to 0,
\]
the piecewise definition is continuous at the only accumulation point \(x_\infty\). Hence
\(F_A\) is a homeomorphism of \(\X\). It preserves \(\mu\), because each piece map
\(\Phi_{j,b_j(A)}\) preserves the restriction of \(\mu\) to \(\Y_j\), and
\(\mu(\{x_\infty\})=0\).

It remains to prove continuity of $\mathcal C$. Let \(\eta>0\). Choose \(J\) such that
\[
        \sup_{j>J}\operatorname{diam}(\Y_j)<\eta.
\]
If \(A,A'\in\Omega_1\) satisfy
\[
        b_j(A)=b_j(A')
\]
for $1\le j\le J$, then
\[
        F_A=F_{A'} \quad\text{on}\quad \{x_\infty\}\cup\bigcup_{j=1}^{J} \Y_j.
\]
On the remaining blocks \(\Y_j\), \(j>J\), both \(F_A(x)\) and \(F_{A'}(x)\) belong to
the same block \(\Y_j\). Hence
\[
        d(F_A(x),F_{A'}(x)) \le \operatorname{diam}(\Y_j) <\eta.
\]
Therefore
\[
        \sup_{x\in \X} d(F_A(x),F_{A'}(x))<\eta.
\]
This is precisely continuity of \(A\mapsto F_A\) from the product topology on
\(\Omega_1\) to the uniform topology on \(C(\X,\X)\).
\end{proof}

\begin{lemma}[Exact point spectrum of a finite clopen cycle]\label{lem:finite-clopen-cycle-point-spectrum}
Let \((Z,\nu)\) be a finite measure space, and suppose
\[
Z=Z_\ast\sqcup Z_{\mathrm{id}}
\]
modulo null sets, where
\[
Z_\ast=Z_0\sqcup\cdots\sqcup Z_{L-1}
\]
with \(0<\nu(Z_r)<\infty\). Suppose \(F(Z_r)=Z_{r+1\!\!\pmod L}\) modulo null sets, \(F^L=\mathrm{id}\) modulo null sets on \(Z_\ast\), and
\(F=\mathrm{id}\) modulo null sets on \(Z_{\mathrm{id}}\). Then
\[
\sigma_p(\K_F : L^\infty(Z)\to L^\infty(Z))=U_L, \quad U_L:=\{\lambda\in\mathbb T:\lambda^L=1\}.
\]
\end{lemma}

\begin{proof}
First consider the pure \(L\)-cycle part \(Z_\ast\). On \(L^\infty(Z_\ast)\) we have $\K_F^L=I$. If \(\lambda^L\ne1\), then
\[
        (\K_F-\lambda I)\left(\K_F^{L-1}+\lambda \K_F^{L-2}+\cdots+\lambda^{L-1}I\right)=(1-\lambda^L)I,
\]
so \(\K_F-\lambda I\) is invertible on the pure cycle part. Hence no such \(\lambda\) is a point eigenvalue on \(Z_\ast\).

If \(\lambda^L=1\), define
\[
        f|_{Z_r}:=\lambda^{-r}
\]
for $r=0,\ldots,L-1$. Then \(f\in L^\infty(Z_\ast)\), \(f\ne0\), and
\[
        \K_F f=\lambda^{-1}f.
\]
Since \(U_L\) is closed under inversion, every element of \(U_L\) occurs as an eigenvalue.

On the identity part \(Z_{\mathrm{id}}\), the point spectrum is contained in \(\{1\}\).
Since \(1\in U_L\), adjoining the identity part does not add anything outside \(U_L\).
Thus
\[
        \sigma_{\mathrm p}(\K_F:L^\infty(Z) \to L^\infty(Z))=U_L.
\]
\end{proof}

\begin{lemma}[Exact point spectrum of an \(\ell^\infty\) block product]\label{lem:Linf-point-block-product}
Let
\[
        \X=\bigsqcup_{j\in J} \X_j
\]
be a countable measurable partition as in \cref{lem:block-lower-norm-exact}, and
suppose that \(F\) leaves each \(\X_j\) invariant modulo null sets. Then
\[
\sigma_{\mathrm p}(\K_F:L^\infty(\X) \to L^\infty(\X))=
\bigcup_{j\in J} \sigma_{\mathrm p}(\K_{F|\X_j} : L^\infty(\X_j) \to L^\infty(\X_j)).
\]
\end{lemma}

\begin{proof}
Under the restriction isomorphism
\[
        L^\infty(\X,\omega)\cong \prod_{j\in J}^{\ell^\infty} L^\infty(\X_j,\omega|_{\X_j}),
\]
the operator \(\K_F\) is the \(\ell^\infty\)-product of the block operators \(\K_{F|\X_j}\).

If \(\lambda\) is a point eigenvalue on one block, extend a nonzero block eigenfunction
by zero to all other blocks. This gives a nonzero \(L^\infty(\X)\)-eigenfunction for \(\K_F\).

Conversely, suppose
\[
        \K_F f=\lambda f
\]
for some nonzero
\[
        f=(f_j)_{j\in J} \in \prod_{j\in J}^{\ell^\infty} L^\infty(\X_j).
\]
Since \(f\ne 0\), there exists \(j_0\in J\) with \(f_{j_0}\ne 0\). The block equation gives
\[
        \K_{F|\X_{j_0}}f_{j_0}=\lambda f_{j_0}.
\]
Hence \(\lambda\) is a point eigenvalue of the \(j_0\)-th block. This proves the equality.
\end{proof}

\begin{theorem}[Spectral decoding of matrix bits]\label{thm:spectral-decoding-matrix-bits}
For every $A\in\Omega_1$ and every $j\in \mathbb{N}$,
\[
b_j(A)=1 \quad\Longleftrightarrow\quad \lambda_j\in\sigp(\K_{F_A} : L^\infty(\X,\mu)\to L^\infty(\X,\mu)).
\]
More precisely,
\[
\sigp(\K_{F_A} : L^\infty\to L^\infty) = \{1\}\cup\bigcup_{\{j : b_j(A)=1\}} U_{q_j}.
\]
\end{theorem}

\begin{proof}
Use the decomposition
\[
        \X=\{x_\infty\}\sqcup\bigsqcup_{j\in \mathbb{N}} \Y_j.
\]
The singleton \(\{x_\infty\}\) has \(\mu\)-measure zero and does not affect \(L^\infty(\X,\mu)\).

Fix \(j\in \mathbb{N}\). If \(b_j(A)=0\), then
\[
        F_A|_{\Y_j}=\mathrm{id}_{\Y_j},
\]
so the \(j\)-th block contributes only the eigenvalue \(1\).

If \(b_j(A)=1\), then \(F_A|_{\Y_j}\) consists of one \(q_j\)-cycle on the clopen
cylinders
\[
        C_{j,0},\ldots,C_{j,q_j-1}
\]
and is the identity on the remaining clopen part of \(\Y_j\). By \cref{lem:finite-clopen-cycle-point-spectrum}, the point spectrum on this block is $U_{q_j}$.
Applying \cref{lem:Linf-point-block-product} to the countable block product gives
\[
        \sigma_{\mathrm p}(\K_{F_A}: L^\infty \to L^\infty) = \{1\}\cup \bigcup_{\{j : b_j(A)=1\}} U_{q_j}.
\]
Now \(q_j\) are pairwise distinct odd primes, and
\[
        \lambda_j=\exp(2\pi i/q_j)
\]
is a primitive \(q_j\)-th root of unity. Hence
\[
        \lambda_j\notin U_{q_k}
\]
for $k\ne j$, and \(\lambda_j\ne1\). Therefore
\[
        \lambda_j\in\sigma_{\mathrm p}(\K_{F_A}) \quad\Longleftrightarrow\quad b_j(A)=1.
\]
This proves both assertions.
\end{proof}

\begin{definition}[Alternating point-spectral Koopman decision problems]\label{def:Theta-m-alternate}
For each $m\in \mathbb{N}$, fix a bijection
\[
\beta_m:\N^m\to\N^2.
\]
Let $Q_1=\exists$, $Q_2=\forall$, $Q_3=\exists, \dots$ . For $F_A\in\K_{\mathrm{pt}}$, define $\Theta_m(F_A)=1$ if and only if
\[
(Q_1 n_1)(Q_2 n_2)\cdots(Q_m n_m) \left[\lambda_{\ell(\beta_m(n_1,\ldots,n_m))}\in\sigp(\K_{F_A} : L^\infty\to L^\infty)\right].
\]
Equivalently, by \cref{thm:spectral-decoding-matrix-bits},
\[
\Theta_m(F_A)=1 \quad\Longleftrightarrow\quad
(Q_1 n_1)\cdots(Q_m n_m) \left[A(\beta_m(n_1,\ldots,n_m))=1\right].
\]
Let \(d_{\mathrm{disc}}\) be the discrete metric on \(\{0,1\}\). The corresponding
Koopman point-spectral decision problem is
\[
        \mathcal P_m^{Kpt} := \left(\Theta_m, \, \mathcal K_{\mathrm{pt}},\,(\{0,1\},d_{\mathrm{disc}}),\,\Lambda_{\mathrm{pt}}\right).
\]
\end{definition}

Let \(\mathcal P_m^{\mathrm{mat}}\) denote the Cantor-matrix source problem from \cite[Def.~9.1-9.2]{sorg2026foundational}. We will prove below that \(\mathcal P_m^{Kpt}\) is finite-query equivalent to \(\mathcal P_m^{\mathrm{mat}}\).

\begin{definition}[Finite-query equivalence in this section]\label{def:local-fq-equivalence}
For the two computational problems \(\mathcal P_m^{\mathrm{mat}}\) and \(\mathcal P_m^{Kpt}\), we write
\[
        \mathcal P_m^{\mathrm{mat}}\equiv_{\mathrm{fq}} \mathcal P_m^{Kpt}
\]
if every matrix-coordinate query can be simulated by finitely many point-evaluation queries on \(\mathcal K_{\mathrm{pt}}\), and every point-evaluation query on
\(\mathcal K_{\mathrm{pt}}\) can be simulated by finitely many matrix-coordinate queries, with the target predicates identified under the coding \(A\mapsto F_A\).
\end{definition}

\begin{proposition}[Finite-query transport between the matrix source and the Koopman family]\label{prop:finite-query-transport-matrix-koopman}
The coding $A\mapsto F_A$ has the following two finite-query properties.
\begin{enumerate}[label=(\roman*),leftmargin=*]
\item \textbf{Coordinate-to-point simulation:} For every matrix coordinate $(a,b)\in\N^2$, let $j=\ell(a,b)$. Then the bit $A(a,b)$ is recovered
from one point query to $F_A$ at the fixed marker $\xi_j$:
\[
A(a,b)=1 \quad\Longleftrightarrow\quad F_A(\xi_j)\ne\xi_j.
\]
\item \textbf{Point-to-coordinate locality:}
For every $x\in \X$ there exists a set $J(x)\subseteq\N$ with $|J(x)|\le1$ such that the value
$F_A(x)$ is determined by the bits $\{b_j(A) : j\in J(x)\}$. Explicitly, $J(x)=\varnothing$ if
$x=x_\infty$ or if $x\in \Y_j\setminus\bigcup_rC_{j,r}$, and $J(x)=\{j\}$ if $x\in\bigcup_r C_{j,r}\subset \Y_j$ for the unique $j$ with $x \in \Y_j$.
\end{enumerate}
Consequently,
\[
        \mathcal P_m^{\mathrm{mat}}\equiv_{\mathrm{fq}} \mathcal P_m^{Kpt}
\]
for every \(m\in \mathbb{N}\), in the raw finite-query sense.
\end{proposition}

\begin{proof}
For (i), let \((a,b)\in\mathbb N^2\) and set
\[
        j:=\ell(a,b).
\]
Then
\[
        b_j(A)=A(\eta(j))=A(a,b).
\]
By construction,
\[
        \Phi_{j,0}(\xi_j)=\xi_j,
\]
whereas
\[
        \Phi_{j,1}(\xi_j)= H_{\alpha_{j,0},\alpha_{j,1}}(\xi_j) \ne \xi_j.
\]
Thus
\[
        A(a,b)=1 \quad\Longleftrightarrow\quad F_A(\xi_j)\ne\xi_j.
\]
Since the point-evaluation oracle returns
\[
        \operatorname{ev}_{\xi_j}(F_A)=\iota_X(F_A(\xi_j)),
\]
and \(\iota_\X\) is injective, comparison with the known value \(\iota_\X(\xi_j)\) recovers
the matrix bit from one point query.

For (ii), fix \(x\in \X\). If \(x=x_\infty\), then
\[
        F_A(x)=x_\infty
\]
for all \(A\), so no matrix bit is needed. If \(x\in \Y_j\setminus\bigcup_r C_{j,r}\),
then \(F_A(x)=x\) for both values of \(b_j(A)\), so again no bit is needed. If
\[
        x\in\bigcup_{r=0}^{q_j-1} C_{j,r}\subseteq \Y_j,
\]
then \(F_A(x)\) is determined completely by the single bit \(b_j(A)\). Hence
\[
        J(x)=
        \begin{cases}
        \varnothing,&x=x_\infty\text{ or }x\in \Y_j\setminus\bigcup_r C_{j,r},\\
        \{j\},&x\in\bigcup_r C_{j,r}\subseteq \Y_j
        \end{cases}
\]
has the required property.

Now consider a finite point-evaluation transcript
\[
        F_A(x_1),\ldots,F_A(x_N).
\]
By the preceding paragraph, each value \(F_A(x_i)\) is determined by at most one bit
\(b_j(A)\), hence by at most one matrix-coordinate query. Therefore the whole point
transcript is simulated by finitely many matrix-coordinate queries.

Conversely, by (i), every matrix-coordinate query is simulated by one fixed marker-point
query. These two simulations convert deepest-level general algorithms in one problem
into deepest-level general algorithms in the other problem, preserving tower height.
Thus the two problems are finite-query equivalent.
\end{proof}

\begin{theorem}[Exact finite type-\(G\) height of the Koopman point-spectrum calibration family]\label{thm:Kpt-exact-height}
For every \(m\in \mathbb{N}_0\),
\[
        \mathrm{SCI}_G(\mathcal P_m^{Kpt})=m.
\]
\end{theorem}

\begin{proof}
By \cite[Thm.~9.7]{sorg2026foundational},
\[
        \mathrm{SCI}_G(\mathcal P_m^{\mathrm{mat}})=m.
\]
By \cref{prop:finite-query-transport-matrix-koopman},
\[
        \mathcal P_m^{\mathrm{mat}} \equiv_{\mathrm{fq}} \mathcal P_m^{Kpt}.
\]
Finite-query equivalence preserves raw type-\(G\) tower height: each deepest-level
general algorithm for one problem is converted into a deepest-level general algorithm for
the other, and the same transformation is applied at every multi-index of the tower.
Therefore
\[
        \mathrm{SCI}_G(\mathcal P_m^{Kpt}) =
        \mathrm{SCI}_G(\mathcal P_m^{\mathrm{mat}}) = m.
\]
\end{proof}

\begin{definition}[Tagged disjoint union of the calibration problems]\label{def:tagged-Kpt-union}
Let
\[
        \Omega_{\mathrm{tag}} := \{(m,F_A) : m\in\mathbb N,\ A\in\Omega_1\}.
\]
Equip this input class with the oracle
\[
        \Lambda_{\mathrm{tag}}:= \{\tau\}\cup\{\operatorname{ev}_x : x\in \X\},
\]
where
\[
        \tau(m,F_A):=m \in \mathbb{C}
\]
and
\[
        \operatorname{ev}_x(m,F_A):=\iota_\X(F_A(x)).
\]
Define
\[
        \Theta_\infty(m,F_A):=\Theta_m(F_A).
\]
The tagged Koopman calibration problem is
\[
        \mathcal P_\infty^{Kpt} :=\left(\Theta_\infty,\,\Omega_{\mathrm{tag}},\,(\{0,1\},d_{\mathrm{disc}}),\,\Lambda_{\mathrm{tag}}\right).
\]
\end{definition}

\begin{theorem}[Infinite type-G height by unbounded finite heights]\label{thm:tagged-infinity-calibration}
The tagged Koopman point-spectrum calibration problem satisfies
\[
\SCIG(\mathcal P_\infty^{\mathrm{Kpt}})=\infty.
\]
\end{theorem}

\begin{proof}
Suppose, toward a contradiction, that \(\mathcal P_\infty^{Kpt}\) is computed by a type-\(G\) tower of finite height \(k\). Restrict the tower to the slice
\[
        \Omega_{k+1}:= \{(k+1,F_A) : A\in\Omega_1\}.
\]
On this slice, the tag query \(\tau\) is constant with value \(k+1\). The remaining
queries are precisely point-evaluation queries of \(F_A\). Hence the restricted tower is
a height-\(k\) tower computing \(\Theta_{k+1}\) on \(\mathcal K_{\mathrm{pt}}\), i.e. it computes
\(\mathcal P_{k+1}^{Kpt}\) at height \(k\). This contradicts \cref{thm:Kpt-exact-height}, which gives
\[
        \SCIG(\mathcal P_{k+1}^{Kpt})=k+1.
\]
Therefore no finite-height tower computes \(\mathcal P_\infty^{Kpt}\), and so
\[
        \SCIG(\mathcal P_\infty^{Kpt})=\infty.
\]
\end{proof}

\begin{remark}[How to avoid an explicit external tag]
If a purely dynamical input class is preferred, the tag $m$ can be encoded into an additional clopen
marker region of the Cantor map. This is only a notational complication: reserve a disjoint sequence
of clopen marker blocks, make the map display a finite cycle of length depending on $m$ on exactly one
marker block, and leave the point-spectrum coding blocks above unchanged. The finite-query proof is
then the same, with one preliminary marker query decoding the tag.
\end{remark}

\section{Outlook And Further Directions}\label{sec:outlook}

The picture separates three endpoint phenomena.

First, the \(L^1\) endpoint fits the residual finite-dimensional paradigm after the target split
\[
        R_{\mathrm{ap},\varepsilon},\,
        C_{\mathrm{ap},\varepsilon},\,
        \sigma_{\mathrm{ap}}.
\]
The proof is an endpoint adaptation of the continuous-dictionary quadrature-residual method used in \cite{sorg2025solvability}.

Second, the \(L^\infty\) approximate point spectrum is unstable in a much stronger
sense. Even arbitrarily small uniform perturbations of a measure-preserving Cantor homeomorphism can change
\[
        \sigma_{\mathrm{ap}}(\K_F : L^\infty\to L^\infty)
\]
from \(\{1\}\) to the whole unit circle. The finite-period Silver-tree route also collapses
to the Borel condition that the selected periods are unbounded, so it cannot produce
analytic hardness for the \(L^\infty\) approximate point spectrum.

Third, the exact \(L^\infty\) point spectrum gives a clean calibration family for raw SCI
height.  The labelled eigenvalue decisions constructed above are finite-query equivalent
to the canonical Cantor-matrix source predicates.  Hence Koopman point-spectrum
membership can realize every finite raw type-\(G\) height, and the tagged union realizes
infinite raw type-\(G\) height.

Several questions remain open. The present work does not give a complete classification of the \(L^\infty\) approximate Koopman point spectrum. It gives no-go diagnostics
for a finite-period Silver construction and a positive calibration theorem for the exact point spectrum. A natural next step seems to seek non-periodic or non-Borel
\(L^\infty\) approximate-spectrum constructions, if such constructions exist, while respecting the distinction between raw SCI, regularity-restricted SCI, and
Weihrauch/Type-2 computability developed in \cite{sorg2026foundational}.

\textbf{Acknowledgments } The author thanks Gustav Conradie for interesting discussions about the work in \cite{sorg2025solvability}, from which the author gained some inspiration for many of the results in this work.

\textbf{Funding } This research did not receive any specific grant from funding agencies in the public, commercial, or not-for-profit sectors.

\textbf{Statement } During the preparation of this work the author used chatGPT4o (through FAU's HAWKI) in order to improve the language in the abstract and introduction. After using this tool, the author reviewed and edited the content as needed and takes full responsibility for the content of the published article.

\textbf{Conflict of interest } The author declares no conflict of interest.

    \bibliographystyle{alpha}  
    \bibliography{bib.bib}

\end{document}